\documentclass{article}
\usepackage {amsmath, amsfonts, amssymb, mathrsfs, amsthm,stmaryrd, sectsty,hyphenat,enumitem}
\usepackage{palatino}
\usepackage[all]{xy}
\newcommand{\bbmu}{\mu}

\def\mf#1{\mathfrak{#1}}
\def\mc#1{\mathcal{#1}}
\def\mb#1{\mathbb{#1}}
\def\tx#1{\textrm{#1}}
\def\tb#1{\textbf{#1}}

\def\R{\mathbb{R}}

\def\C{\mathbb{C}}
\def\Q{\mathbb{Q}}

\def\Z{\mathbb{Z}}
\def\N{\mathbb{N}}

\def\ol#1{\overline{#1}}

\def\hat{\widehat}

\def\rw{\rightarrow}
\def\lw{\leftarrow}

\def\lrw{\longrightarrow}

\def\lw{\leftarrow}

\def\<{\langle}
\def\>{\rangle}

\def\pf{\tb{Proof: }}

\newenvironment{mytitle}
{\begin{center}\large\sc}
{\end{center}}

\newtheorem{thm}{Theorem}[section]
\newtheorem{lem}[thm]{Lemma}
\newtheorem{pro}[thm]{Proposition}
\newtheorem{cor}[thm]{Corollary}
\newtheorem{dfn}[thm]{Definition}
\newtheorem{fct}[thm]{Fact}

\sectionfont{\center\sc\normalsize}
\subsectionfont{\bf\normalsize}

\numberwithin{equation}{section}

\begin{document}

\begin{mytitle} Rigid inner forms vs isocrystals \end{mytitle}
\begin{center} Tasho Kaletha \end{center}

\begin{abstract}
	We compare two statements of the refined local Langlands correspondence for connected reductive groups defined over a $p$-adic field -- one involving Kottwitz's set $B(G)$ of isocrystals with additional structure, and one involving the cohomology set $H^1(u \rw W,Z \rw G)$ of \cite{KalRI}. We show that if either statement is valid for all connected reductive groups, then so is the other. We also discuss how the second statement depends on the choice of element of $H^1(u \rw W,Z \rw G)$.
\end{abstract}
{\let\thefootnote\relax\footnotetext{This research is supported in part by NSF grant DMS-1161489.}}

\tableofcontents

\section{Introduction}
The basic form of the local Langlands conjecture predicts a correspondence between Langlands parameters $\varphi$ for a given connected reductive group $G'$ defined over a local field $F$ and finite sets $\Pi_\varphi(G')$ of irreducible admissible representations of the topological group $G'(F)$. Refinements of this conjecture give a description of the elements of $\Pi_\varphi$. When the group $G'$ is not quasi-split, these refinements involve the choice of realization of $G'$ as an inner form of a quasi-split group $G$, as well as further objects of Galois-cohomological nature. More precisely, one fixes a quasi-split group $G$ and an inner twist $\psi : G \rw G'$. Then $\sigma \mapsto \psi^{-1}\sigma(\psi)$ is a 1-cocycle of the absolute Galois group of $F$ with values in the adjoint group of $G$. To state the refined local Langlands conjecture, one needs to fix a lift of this 1-cocycle to a 1-cocycle of a certain modification of the Galois group of $F$ with values in $G$, rather than its adjoint quotient.

One statement of the refined local Langlands conjecture uses Kottwitz's set $B(G)$ of isocrystals with $G$-structure \cite{Kot85}, \cite{Kot97}, \cite{KotBG}, and in particular the subset $B(G)_\tx{bas}$ of basic isocrystals, to provide a lift $[x_\tx{iso}] \in B(G)_\tx{bas}$ of $\psi^{-1}\sigma(\psi)$. This statement will be referred to as $\tx{LLC}_\tx{iso}(\psi,x_\tx{iso})$ in this paper (for the purposes of the introduction, we will be vague about the difference between the 1-cocycle $x_\tx{iso}$ and its cohomology class $[x_\tx{iso}]$). We will denote by $\tx{LLC}_\tx{iso}$ the totality of all statements $\tx{LLC}_\tx{iso}(\psi,[x_\tx{iso}])$ for all possible quasi-split groups $G$, inner twists $\psi : G \rw G'$, and lifts $[x_\tx{iso}] \in B(G)_\tx{bas}$ of $\psi^{-1}\sigma(\psi)$.

The statement $\tx{LLC}_\tx{iso}(\psi,x_\tx{iso})$ is formulated in \cite[\S2.4]{KalIso} for discrete parameters (see also \cite[\S5]{Rap95}), and in \cite[\S1.6.1]{KMSW} for tempered parameters of unitary groups. A general formulation can be found in \cite[\S2.5]{KalSimons}. The set $B(G)$ occurs naturally in the study of Shimura varieties and Rapoport-Zink spaces. A conjecture of Kottwitz \cite[Conjecture 5.1]{Rap95} describes the contribution of cuspidal $L$-packets to the cohomology of Rapoport-Zink spaces in terms of the parameterization given by $\tx{LLC}_\tx{iso}$. This makes $\tx{LLC}_\tx{iso}$ well suited for the study of these geometric objects and conversely hints at the possibility of finding a proof of $\tx{LLC}_\tx{iso}$ by studying Rapoport-Zink spaces and their generalizations. A conjectural program for this was recently announced by Fargues \cite{FarMSRI}, building on his description of vector bundles on the Fargues-Fontaine curve \cite{FFC} and ideas of Scholze. At the same time, this statement of the local Langlands conjecture has the disadvantage of not being available for all connected reductive groups, because the 1-cocycle $\psi^{-1}\sigma(\psi)$ may fail to lift to an element of $B(G)_\tx{bas}$. This is not a problem when the center of $G$ is connected, but it is a significant problem when $G$ is simply connected, for example. A further disadvantage of $\tx{LLC}_\tx{iso}$ is that it is unclear how it relates to Arthur's work on the stabilization of the trace formula for groups which do not satisfy the Hasse principle.

Another statement of the refined local Langlands conjecture uses the cohomology set $H^1(u \rw W,Z(G) \rw G)$ defined in \cite{KalRI} to provide a lift $[x_\tx{rig}]$ of $\psi^{-1}\sigma(\psi)$. This statement, which we will call $\tx{LLC}_\tx{rig}(\psi,x_\tx{rig})$ in this paper, is available for all connected reductive groups (without assumption on the center). There is an explicit connection between it and the stabilization of the Arthur-Selberg trace formula \cite{KalGRI}. It has been furthermore shown \cite[\S5.6]{KalRI} that when the ground field is $\R$ this statement is true and that the set $H^1(u \rw W,Z(G) \rw G)$ is in canonical bijection with the set of equivalence classes of strong real forms of $G$ due to Adams-Barbasch-Vogan \cite{ABV92}. However, it is not clear how $\tx{LLC}_\tx{rig}$ relates to the cohomology of Rapoport-Zink spaces.

Our main goal in this paper is to compare the statements $\tx{LLC}_\tx{iso}$ and $\tx{LLC}_\tx{rig}$, thereby building a bridge between the stable Arthur-Selberg trace formula and the cohomology of Rapoport-Zink spaces, and in particular Fargues' conjectural program. We expect that this bridge will be useful in both ways. In one direction, it will facilitate applications of the trace formula to the study of Shimura varieties and their local analogs. In the other direction, it will transfer potential results of Fargues' program to the setting of the trace formula and also to the setting of arbitrary connected reductive groups without assumption on their center.

The comparison of the two statements is based on a comparison of the cohomology sets $B(G)_\tx{bas}$ and $H^1(u \rw W,Z(G) \rw G)$. The set $B(G)$ was initially defined as a set of Frobenius-twisted conjugacy classes in the group $G$ \cite{Kot85} and was later reinterpreted as the cohomology of a certain Galois gerbe with values in $G$, initially in the case of tori in \cite[\S8]{Kot97}, and then later in general in \cite{KotBG}. The set $H^1(u \rw W,Z(G) \rw G)$ was defined directly using Galois gerbes \cite[\S3]{KalRI}. The two Galois gerbes underlying $B(G)$ and $H^1(u \rw W,Z(G) \rw G)$ are of quite different nature. The one for $B(G)$ is bound by a split pro-torus, while the one for $H^1(u \rw W,Z(G) \rw G)$ is bound by a pro-finite multiplicative group which is far from being split and whose character module encodes the arithmetic of the ground field $F$. For this reason, we did not initially expect that there can be any reasonable comparison between the two. However, it turns out that a certain universal property of the pro-finite multiplicative group that binds the gerbe of \cite{KalRI} is responsible for the existence of an essentially unique homomorphism between the two gerbes. This homomorphism leads in turn to a comparison map $B(G)_\tx{bas} \rw H^1(u \rw W,Z(G) \rw G)$. The comparison map is in general neither injective, nor surjective. For example we have $B(\mb{G}_m)_\tx{bas}=\Z$ and $H^1(u \rw W,\mb{G}_m \rw \mb{G}_m)=0$, while for a 1-dimensional anisotropic torus $S$ we have $B(S)_\tx{bas}=H^1(\Gamma,S)=\Z/2\Z$ and $H^1(u \rw W,S \rw S)=\Q/2\Z$. In general, there is a simple description of the comparison map $B(G)_\tx{bas} \rw H^1(u \rw W,Z(G) \rw G)$ in terms of generalized Tate-Nakayama duality. This description plays a central role in the comparison between $\tx{LLC}_\tx{iso}$ and $\tx{LLC}_\tx{rig}$.

We will now describe the structure of this paper. The reader who wishes a more general introduction to the different statements of the refined local Langlands conjecture and the problems presented by non-quasi-split groups might find the survey \cite{KalSimons} useful. The comparison of the cohomology sets $B(G)_\tx{bas}$ and $H^1(u \rw W,Z(G) \rw G)$ is done in Section \ref{sec:compgerb}. In Section \ref{sec:compllc} we briefly recall the statements $\tx{LLC}_\tx{iso}(\psi,x_\tx{iso})$ and $\tx{LLC}_\tx{rig}(\psi,x_\tx{rig})$ for a fixed inner twist $\psi : G \rw G'$ of a quasi-split connected reductive group $G$. They depend on the choice of a lifts $[x_\tx{iso}] \in B(G)_\tx{bas}$ and $[x_\tx{rig}] \in H^1(u \rw W,Z(G) \rw G)$ of the class of $\psi^{-1}\sigma(\psi)$. The lift $[x_\tx{iso}]$ will not always exist, but in Section \ref{sec:compllc} we make the assumption that it does. We then give a comparison between $\tx{LLC}_\tx{iso}(\psi,x_\tx{iso})$ and $\tx{LLC}_\tx{rig}(\psi,x_\tx{rig})$. The backbone of this comparison is the map $B(G)_\tx{bas} \rw H^1(u \rw W,Z(G) \rw G)$ studied in \ref{sec:compgerb}. The comparison is given by an explicit formula and can be calculated explicitly for any given example. A consequence of this comparison is that the statements $\tx{LLC}_\tx{iso}(\psi,x_\tx{iso})$ and $\tx{LLC}_\tx{rig}(\psi,x_\tx{rig})$ are equivalent.

In particular, once $\tx{LLC}_\tx{rig}(\psi,x_\tx{rig})$ is proved for all $\psi$ and all $[x_\tx{rig}]$, it implies $\tx{LLC}_\tx{iso}(\psi,x_\tx{iso})$ for all $\psi$ and all $x_\tx{iso}$. This establishes the implication $\tx{LLC}_\tx{rig} \Rightarrow \tx{LLC}_\tx{iso}$. We also want to obtain the converse implication, but the fact that $\psi^{-1}\sigma(\psi)$ doesn't always lift to $B(G)_\tx{bas}$ necessitates further work, which is done in Sections \ref{sec:redcent} and \ref{sec:changrig}.

We alert the reader familiar with the statement $\tx{LLC}_\tx{iso}$ from other references that in this paper we are using a slight modification of this statement. This modification is only visible for groups with disconnected center and for non-discrete tempered parameters. In particular, the statement of Kottwitz's conjecture \cite[Conjecture 5.1]{Rap95} or the results of \cite{KMSW} are not affected.

In Section \ref{sec:redcent} we deal with the problem that the comparison of Section \ref{sec:compllc} was done under the assumption that $\psi^{-1}\sigma(\psi)$ lifts to $B(G)_\tx{bas}$. This assumption cannot be removed and the way we deal with it is necessarily roundabout. Namely, given a connected reductive group $G$ we introduce a procedure that embeds $G$ into a connected reductive group $G_z$ with connected center and comparable representation theory and endoscopy. The idea for this procedure is due to Kottwitz and was communicated verbally to the author some years ago. This procedure, which we call $z$-embedding, is formalized and generalized in Subsection \ref{sub:zemb}, where we also study its implications to representations, endoscopy, and inner twistings. In particular, to any inner twist $\psi : G \rw G'$ there is an associated inner twist $\psi_z : G_z \rw G'_z$ of $z$-embeddings. The natural map $H^1(u \rw W,Z(G) \rw G) \rw H^1(u \rw W,Z(G_z) \rw G_z)$ is bijective. Let $[x_\tx{rig}] \in H^1(u \rw W,Z(G) \rw G)$ lift the class of $\psi^{-1}\sigma(\psi)=\psi_z^{-1}\sigma(\psi_z)$. We show that the statements $\tx{LLC}_\tx{rig}(\psi,x_\tx{rig})$ and $\tx{LLC}_\tx{rig}(\psi_z,x_\tx{rig})$ are equivalent.

A consequence of the results of Sections \ref{sec:compllc} and \ref{sec:redcent} is the following. Let $[x_\tx{iso}] \in B(G_z)_\tx{bas}$ be a lift of $\psi^{-1}\sigma(\psi)=\psi_z^{-1}\sigma(\psi_z)$. It exists since $Z(G_z)$ is connected. Let $[x_\tx{rig}] \in H^1(u \rw W,Z(G_z) \rw G_z) = H^1(u \rw W,Z(G) \rw G)$ be its image under the comparison map $B(G_z)_\tx{bas} \rw H^1(u \rw W,Z(G_z) \rw G_z)$. Then the statements $\tx{LLC}_\tx{iso}(\psi_z,x_\tx{iso})$ and $\tx{LLC}_\tx{rig}(\psi,x_\tx{rig})$ are equivalent. In other words, once the statement $\tx{LLC}_\tx{iso}(\tilde\psi,\tilde x_\tx{iso})$ is proved for every inner twist $\tilde\psi : \tilde G \rw \tilde G'$ of a quasi-split connected reductive group $\tilde G$ with connected center, and some $[x_\tx{iso}] \in B(\tilde G)_\tx{bas}$ lifting $\psi^{-1}\sigma(\psi)$, then this implies the validity of the statement $\tx{LLC}_\tx{rig}(\psi,x_\tx{rig})$ for every inner twist $\psi : G \rw G'$ of a quasi-split connected reductive group $G$, without assumption on the center, and some $[x_\tx{rig}] \in H^1(u \rw W,Z(G) \rw G)$ lifting $\psi^{-1}\sigma(\psi)$.

In order to complete the proof of the implication $\tx{LLC}_\tx{iso} \Rightarrow \tx{LLC}_\tx{rig}$, we must now show that if $[x_{1,\tx{rig}}],[x_{2,\tx{rig}}] \in H^1(u \rw W,Z(G) \rw G)$ both lift the class of $\psi^{-1}\sigma(\psi)$, then the statements $\tx{LLC}_\tx{rig}(\psi,x_{1,\tx{rig}})$ and $\tx{LLC}_\tx{rig}(\psi,x_{2,\tx{rig}})$ are equivalent. This is done in Section \ref{sec:changrig}. We give an explicit relationship between $\tx{LLC}_\tx{rig}(\psi,x_{1,\tx{rig}})$ and $\tx{LLC}_\tx{rig}(\psi,x_{2,\tx{rig}})$. The experience of \cite{KMSW} has shown that such a relationship is of interest in its own right. For example, it would be useful when one proves $\tx{LLC}_\tx{rig}(\psi,x_{\tx{rig}})$ using the trace formula and the local-global methods of \cite{KalGRI}.

\section{Notation}
Throughout this paper, $F$ will denote a $p$-adic field, i.e. a finite extension of the field $\Q_p$ of $p$-adic numbers. We fix an algebraic closure $\ol{F}$ of $F$ and let $\Gamma$ denote the absolute Galois group of $\ol{F}/F$ and $W_F$ the absolute Weil group. We will write $L_F$ for the Langlands group of $F$, which we can interpret either as the product $W_F \times \tx{SL}_2(\C)$ or $W_F \times \tx{SU}_2(\R)$.

\section{Comparison of the cohomology of two Galois gerbes} \label{sec:compgerb}
The purpose of this section is to construct for any affine algebraic group $G$ a comparison map $B(G)_\tx{bas} \rw H^1(u \rw W,Z(G) \rw G)$, where $B(G)_\tx{bas}$ is the set of isomorphism classes of basic isocrystals with $G$-structure \cite{Kot85}, and $H^1(u \rw W,Z(G) \rw G)$ is a variant of the cohomology set introduced in \cite{KalRI}. We will first review each of these sets from a point of view that is slightly different than their original definition.

\subsection{Review of $B(G)$}
Let $G$ be an affine algebraic group. In \cite{Kot85}, Kottwitz studies the set $B(G)$ of Frobenius-twisted conjugacy classes of elements of $G(L)$, where $L$ is the completion of the maximal unramified extension of $F$. This set can also be described as the set of continuous cohomology classes of $W_F$ with values in $G(\ol{L})$ \cite[\S1.4]{Kot97}, or as the set of continuous cohomology classes of a certain Galois gerbe with values in $G(\ol{F})$ \cite[\S8]{Kot97}, \cite{KotBG}.

In this subsection we will review the set $B(G)$, as well as a certain subset of it called $B(G)_\tx{bas}$. We will give a slightly different construction, again in terms of Galois gerbes, but closer in spirit to the point of view of \cite{KalRI}.

We begin by recalling the pro-torus $\mb{D}=\mb{D}_F$ defined in \cite{KotBG}. Consider the contravariant functor $\Phi$ from the category of all finite Galois extensions of $F$ contained in $\ol{F}$ with morphisms given by $F$-algebra homomorphisms, to the category of algebraic groups, defined to send every extension $E/F$ to the split one-dimensional torus $\mb{G}_m$ and every homomorphism $E \rw K$ to the $[K:E]$-power map. The pro-torus $\mb{D}$ is defined as the limit of $\Phi$. We claim that the group $H^2(\Gamma,\mb{D}(\ol{F}))$ has a distinguished element. To construct it, we will introduce a variation of the construction of $\mb{D}$ that will be useful later as well. Consider the functor $\Phi'$ between the same categories as $\Phi$, but now defined by $\Phi'(E/F)=\mu_{[E:F]}$ and $\Phi'(E \rw K) = (\ )^{[K:E]} : \mu_{[K:F]} \rw \mu_{[E:F]}$. Let $\bbmu$ be the limit of $\Phi'$. If we identify $H^2(\Gamma,\mu_n(\ol{F})) = \Z/n\Z$ via local class field theory, then according to \cite[Corollary 2.7.6]{NSW08}, we have
\[ H^2(\Gamma,\bbmu(\ol{F}))=\varprojlim_n H^2(\Gamma,\mu_n(\ol{F})) = \hat \Z. \]
We have the obvious map $\bbmu \rw \mb{D}$ and the image of $1 \in \hat \Z$ under this map is the distinguished element of $H^2(\Gamma,\mb{D}(\ol{F}))$.

Let
\begin{equation} \label{eq:bgext}
1 \rw \mb{D}(\ol{F}) \rw \mc{E} \rw \Gamma \rw 1
\end{equation}
be an extension corresponding to the distinguished class. The topological group $\mc{E}$ acts on the discrete group $G(\ol{F})$ via the map $\mc{E} \rw \Gamma$ and we consider the cohomology set $H^1(\mc{E},G(\ol{F}))$. The restriction of an element of this set to $\mb{D}$ is a $\Gamma$-invariant $G(\ol{F})$-conjugacy class of continuous group homomorphisms $\mb{D}(\ol{F}) \rw G(\ol{F})$. The set $B(G)$ is defined to be the subset of $H^1(\mc{E},G(\ol{F}))$ consisting of those classes whose restriction to $\mb{D}$ is given by algebraic homomorphisms $\mb{D} \rw G$. The smaller set $B(G)_\tx{bas}$ consists of those elements of $B(G)$ for which the restriction to $\mb{D}(\ol{F})$ consists of homomorphisms taking image in $Z(G)$. Then the $G(\ol{F})$-conjugacy class of these homomorphisms consists of a single element, and this element if $\Gamma$-fixed, i.e. it is an element of $\tx{Hom}_F(\mb{D},Z(G))$.

We note that, while this construction of $B(G)$ and $B(G)_\tx{bas}$ used a specific choice of the extension $\mc{E}$ within the isomorphism class given by the distinguished element of $H^2(\Gamma,\mb{D}(\ol{F}))$, the result is in fact independent of this choice up to a unique isomorphism. It is clear that if $\mc{E'}$ is another extension in the same isomorphism class and if we fix an isomorphism $f : \mc{E'} \rw \mc{E}$, then composing 1-cocycles with $f$ provides a bijection $H^1(\mc{E},G(\ol{F})) \rw H^1(\mc{E'},G(\ol{F}))$ and this bijection identifies the corresponding versions of $B(G)$ and $B(G)_\tx{bas}$. We claim now that this bijection does not depend on the choice of $f$. A second such isomorphism has the form $f'(e)=f(e) \cdot x(\sigma_e)$, where $x \in Z^1(\Gamma,\mb{D}(\ol{F}))$ and $\sigma_e \in \Gamma$ is the image of $e$. Then $z(f'(e))=z(f(e)\cdot x(\sigma_e)) = z(f(e)) \cdot z(x(\sigma_e))$. For any $z \in Z^1(\mc{E},G(\ol{F}))$, the restriction $z|_{\mb{D}}$ factors through the projection $\mb{D} \rw \Phi(E/F)=\mb{G}_m$ for a suitable finite extension $E/F$. By Hilbert 90 the composition of $x$ with the projection $\mb{D} \rw \mb{G}_m$ is a co-boundary, and hence $e \mapsto z(x(\sigma_e))$ is itself a co-boundary.

Now let $G=S$ be a torus. Then we trivially have $B(S)=B(S)_\tx{bas}$. The restriction of an element of $B(S)$ to the protorus $\mb{D}$ is an element of $\tx{Hom}_F(\mb{D},S)=\tx{Hom}(X^*(S),\Q)^\Gamma=[X_*(S) \otimes \Q]^\Gamma$. Thus restriction provides a map
\begin{equation} \label{eq:bgnewt}
B(S) \rw [X_*(S) \otimes \Q]^\Gamma
\end{equation}
which is sometimes called the Newton map. The kernel of the Newton map is equal to the image of the inflation $H^1(\Gamma,S) \rw B(S)$. Furthermore, Kottwitz constructs a functorial isomorphism
\begin{equation} \label{eq:bgisos}
X_*(S)_\Gamma \rw B(S).
\end{equation}
The composition of this isomorphism with the Newton map is given by
\[ N^\diamond : X_*(S)_\Gamma \rw [X^*(S)\otimes\Q]^\Gamma,\qquad y \mapsto [E:F]^{-1}N_{E/F}(y), \]
where $E/F$ is any finite Galois extension that splits $S$. Altogether we obtain the commutative diagram with exact rows

\begin{equation} \label{eq:bgdiag}
\xymatrix{
	0\ar[r]&H^1(\Gamma,S)\ar[r]&B(S)\ar[r]&\tx{Hom}_F(\mb{D},S)^\Gamma\\
	0\ar[r]&X_*(S)_{\Gamma,\tx{tor}}\ar[r]\ar[u]^{\tx{TN}}&X_*(S)_\Gamma\ar[r]^-{N^\diamond}\ar[u]^{\eqref{eq:bgisos}}&[X_*(S)\otimes\Q]^\Gamma\ar@{=}[u]\\
}
\end{equation}
where $\tx{TN}$ is the Tate-Nakayama isomorphism.
The isomorphism \eqref{eq:bgisos} can be phrased as a duality statement. If $\hat S = X_*(S) \otimes \C^\times$ denotes the complex torus dual to $S$, then $X^*(\hat S^\Gamma)=X_*(S)_\Gamma$ and thus \eqref{eq:bgisos} becomes the duality pairing
\begin{equation} \label{eq:bgduals}
\hat S^\Gamma \otimes B(S) \rw \C^\times.
\end{equation}
This duality in turn generalizes to the case where $G$ is a connected reductive group. In that case, we have the duality
\begin{equation} \label{eq:bgdualg}
Z(\hat G)^\Gamma \otimes B(G)_\tx{bas} \rw \C^\times.
\end{equation}

\subsection{Review of $H^1(u \rw W,Z \rw G)$}
We will now give a short review of the cohomology set $H^1(u \rw W,Z \rw G)$ introduced in \cite{KalRI} for an affine algebraic group $G$ and a finite central subgroup $Z$, both defined over $F$. Consider the functor $\Phi''$ from the category of finite Galois extensions of $F$ contained in $\ol{F}$ to the category of algebraic groups that sends the object $E/F$ to $u_{E/F} = \tx{Res}_{E/F}\mu_{[E:F]}/\mu_{[E:F]}$ and the morphism $E \rw K$ to the map $u_{K/F} \rw u_{E/F}$ that assigns to $f \in u_{K/F}$ the function $\sigma \mapsto \prod_{\tau \mapsto \sigma} f(\tau)^\frac{m}{n}$. Here we are using the interpretation $\tx{Res}_{E/F}\mu_{[E:F]}(\ol{F}) = \tx{Maps}(\Gamma_{E/F},\mu_{[E:F]}(\ol{F}))$. Let $u$ be the limit of this functor. It is a pro-finite algebraic group. According to \cite[Theorem 3.1]{KalRI} we have $H^1(\Gamma,u)=0$ and $H^2(\Gamma,u)=\hat \Z$. Let
\begin{equation} \label{eq:riext}
1 \rw u(\ol{F}) \rw W \rw \Gamma \rw 1
\end{equation}
be an extension corresponding to the element $-1 \in \hat \Z = H^2(\Gamma,u)$. As in the previous subsection we can consider the cohomology set $H^1(W,G(\ol{F}))$ and we define $H^1(u \rw W,Z \rw G)$ to be the subset of those cohomology classes whose restriction to $u(\ol{F})$ takes image in $Z$. Then this restriction is the composition of the natural projection $u(\ol{F}) \rw u_{E/F}(\ol{F})$ for some $E/F$ with a group homomorphism $u_{E/F}(\ol{F}) \rw Z(\ol{F})$. This composition is automatically an element of $\tx{Hom}_F(u,Z)$. The set $H^1(u \rw W,Z \rw G)$ is independent of the choice of $W$ up to a unique isomorphism due to the vanishing of $H^1(\Gamma,u)$.

For the purposes of comparing with $B(G)_\tx{bas}$ we define for any multiplicative central subgroup $Z \subset G$ defined over $F$ (but not necessarily finite)
\begin{equation} \label{eq:h}
H^1(u \rw W,Z \rw G) = \varinjlim_{Z'} H^1(u \rw W,Z' \rw G)
\end{equation}
where $Z'$ runs over all finite subgroups of $Z$ defined over $F$.

Let $S=G$ be a torus and $Z \subset S$ a finite subgroup. Restricting an element of $H^1(u \rw W,Z \rw S)$ to the group $u$ provides an element of $\tx{Hom}(u,Z)^\Gamma=\tx{Hom}(X^*(Z),\Q/\Z)$. There is a functorial isomorphism
\begin{equation} \label{eq:riisos}
[\bar Y/IY]_\tx{tor} \rw H^1(u \rw W,Z \rw S),
\end{equation}
where $\bar Y=X_*(S/Z)$, $Y=X_*(S)$, and $IY=\<\sigma(y)-y|y \in Y,\sigma \in \Gamma\>$. Taking the limit over all finite subgroups $Z$ of $S$ we obtain a restriction map $H^1(u \rw W,S \rw S) \rw \tx{Hom}(u,S)^\Gamma=\tx{Hom}(X^*(S),\Q/\Z)$ and a functorial isomorphism $[Y \otimes \Q/IY]_\tx{tor} \rw H^1(u \rw W,S \rw S)$. Altogether we obtain a commutative diagram with exact rows
\begin{equation} \label{eq:ridiag}
\xymatrix{
	0\ar[r]&H^1(\Gamma,S)\ar[r]&H^1(u \rw W,S \rw S)\ar[r]&\tx{Hom}_F(u,S)\\
	0\ar[r]&\frac{Y}{IY}[\tx{tor}]\ar[u]^{\tx{TN}}\ar[r]&\frac{Y \otimes \Q}{IY}[\tx{tor}]\ar[u]^-{\eqref{eq:riisos}}\ar[r]&\frac{Y \otimes \Q}{Y}\ar@{=}[u]
}
\end{equation}
Again we can phrase the isomorphism \eqref{eq:riisos} as a duality pairing. Indeed, let $\hat{\bar S} \rw \hat S$ be the universal cover of $\hat S$, that is, the projective limit of all tori that are finite covers of $S$. Let $[\hat{\bar S}]^+$ be the preimage in $\hat{\bar S}$ of $\hat S^\Gamma$. We have the chain of subgroups $[\hat{\bar S}]^{+,\circ}=[\hat{\bar S}]^{\Gamma,\circ} \subset [\hat{\bar S}]^\Gamma \subset [\hat{\bar S}]^+$. Then $X^*(\hat{\bar S})=Y \otimes \Q$ and $\frac{Y\otimes\Q}{IY}[\tx{tor}] = X^*(\pi_0([\hat{\bar S}]^+))$. The isomorphism \eqref{eq:riisos} then becomes the duality pairing
\begin{equation} \label{eq:riduals}
\pi_0([\hat{\bar S}]^+) \otimes H^1(u \rw W,S \rw S) \rw \C^\times.
\end{equation}
This duality pairing again generalizes to the case of a connected reductive group $G$. Let $\hat G$ be the complex Langlands dual group of $G$ and let $\hat{\bar G}$ be the projective limit of all central isogenies with target $\hat G$. Defining $Z(\hat{\bar G})^+$ to be the preimage of $Z(\hat G)^\Gamma$ we again obtain the tower of subgroups $Z(\hat{\bar G})^{+,\circ}=Z(\hat{\bar G})^{\Gamma,\circ} \subset Z(\hat{\bar G})^{\Gamma} \subset Z(\hat{\bar G})^+$. We have the duality pairing
\begin{equation} \label{eq:ridualg}
\pi_0(Z(\hat{\bar G})^+) \otimes H^1(u \rw W,Z(G) \rw G) \rw \C^\times.
\end{equation}

\subsection{A comparison map $B(G)_\tx{bas} \rw H^1(u \rw W,Z(G) \rw G)$} \label{sub:cohcomp}
According to \cite[Proposition 3.2]{KalRI} there exists a unique $\phi_n \in \tx{Hom}_F(u,\mu_n)$ with the property that the image of $-1 \in \hat \Z=H^2(\Gamma,u)$ under $\phi_n$ is equal to $1 \in \Z/n\Z = H^2(\Gamma,\mu_n)$. For two natural numbers $n|m$, the composition of $\phi_m$ with $(\ )^\frac{m}{n} : \mu_m \rw \mu_n$ is equal to $\phi_n$. Thus we obtain $\phi \in \tx{Hom}_F(u,\bbmu)$, which sends $-1 \in \hat \Z = H^2(\Gamma,u)$ to $1 \in \hat \Z=H^2(\Gamma,\mu)$. We compose $\phi$ with the obvious map $\bbmu \rw \mb{D}$ and denote the result again by $\phi \in \tx{Hom}_F(u,\mb{D})$. We can then realize the extension \eqref{eq:bgext} as the push-out of the extension \eqref{eq:riext} along $\phi$, i.e.
\begin{equation} \label{eq:compdiag}
\xymatrix{
	1\ar[r]&u\ar[r]\ar[d]^\phi&W\ar[r]\ar[d]&\Gamma\ar[r]\ar@{=}[d]&1\\
	1\ar[r]&\mb{D}\ar[r]&\mc{E}\ar[r]&\Gamma\ar[r]&1\\
}
\end{equation}
Composing 1-cocycles with the homomorphism $W \rw \mc{E}$ provides a map
\begin{equation} \label{eq:compmap}
B(G)_\tx{bas} \rw H^1(u \rw W,Z(G) \rw G).
\end{equation}
Note that when $G$ is a torus, this map is a homomorphism of abelian groups and is moreover functorial. For general $G$, the map is a map of sets. It does not make sense to ask for its functoriality, because the assignments $G \mapsto B(G)_\tx{bas}$ and $G \mapsto H^1(u \rw W,Z(G) \rw G)$ are not functors.

We will now discuss how the map \eqref{eq:compmap} translates under the isomorphisms \eqref{eq:bgisos} and \eqref{eq:riisos}, as well as under the dualities \eqref{eq:bgdualg} and \eqref{eq:ridualg}.

\begin{lem} \label{lem:phiexp}
Let $E/F$ be a finite Galois extension and $n$ a divisor of $[E:F]$. Consider the map
\[ \tx{Res}_{E/F}\mu_{[E:F]} \stackrel{-N_{E/F}}{\lrw} \mu_{[E:F]} \stackrel{(\ )^{[E:F]/n}}{\lrw} \mu_n. \]
This map descends to $u_{E/F}$ and its composition with the natural projection $u \rw u_{E/F}$ equals $\phi_n$.
\end{lem}
\begin{proof}
We have $H^2(\Gamma,\tx{Res}_{E/F}\mu_{[E:F]})=H^2(\Gamma_E,\mu_{[E:F]})=\Z/[E:F]\Z$. The image of $-1 \in \hat\Z = H^2(\Gamma,u)$ in $H^2(\Gamma,u_{E/F})$ is equal to the image of $-1 \in \Z/[E:F]\Z=H^2(\Gamma,\tx{Res}_{E/F}\mu_{[E:F]})$ there. The lemma will be proved once we show that the map in the statement of the lemma maps $-1 \in \Z/[E:F]\Z$ to $1 \in \Z/n\Z=H^2(\Gamma,\mu_n)$.

The composition of the isomorphism $H^2(\Gamma_E,\mu_{[E:F]}) \cong H^2(\Gamma,\tx{Res}_{E/F}\mu_{[E:F]})$ with $N_{E/F}$ is equal to the corestriction map. The composition
\[ \Z/[E:F]\Z = H^2(\Gamma_E,\mu_{[E:F]}) \stackrel{\tx{cor}}{\lrw} H^2(\Gamma,\mu_{[E:F]}) = \Z/[E:F]\Z, \]
where we have used the local reciprocity maps for the fields $E$ and $F$, respectively, is equal to the identity. This completes the proof.
\end{proof}

\begin{pro} \label{pro:comps} Let $S$ be a torus. The composition of \eqref{eq:compmap} for $G=S$ with the isomorphisms \eqref{eq:bgisos} and \eqref{eq:riisos} is given by
\[ \frac{Y}{IY} \rw \frac{Y \otimes \Q}{IY}[\tx{tor}],\qquad y \mapsto y-N^\diamond(y). \]
\end{pro}
\begin{proof}
Let us denote by $\alpha$ the composition we are studying, and by $\beta$ the displayed map. Both of these are functorial homomorphisms and our goal is to show that they are equal. Taking a look at diagrams \eqref{eq:bgdiag} and \eqref{eq:ridiag} we note that both $\alpha$ and $\beta$ identify the copies of $\frac{Y}{IY}[\tx{tor}]$ embedded into their source and target. This leads to the diagram
\[ \xymatrix{
	0\ar[r]&\frac{Y}{IY}[\tx{tor}]\ar[r]\ar@{=}[d]&\frac{Y}{IY}\ar[r]^-{N^\diamond}\ar@<1ex>[d]^\beta\ar@<-1ex>[d]_\alpha&[Y\otimes\Q]^\Gamma\ar[d]^\gamma\\
	0\ar[r]&\frac{Y}{IY}[\tx{tor}]\ar[r]&\frac{Y \otimes \Q}{IY}[\tx{tor}]\ar[r]&\frac{Y \otimes \Q}{Y}
}\]
We claim that if $\gamma$ is given by multiplication by $-1$, followed by the inclusion $[Y \otimes \Q]^\Gamma \rw Y \otimes \Q$, followed by the projection $Y \otimes \Q \rw Y \otimes \Q/Y$, then the diagram commutes with both $\alpha$ and $\beta$. In the case of $\beta$ this is obvious, because the image of $\beta(y)=y-N^\diamond(y)$ in $Y\otimes \Q/Y$ is equal to $-N^\diamond(y)$ and thus coincides with the image of $N^\diamond(y)$ under $\gamma$. In the case of $\alpha$ we take $y \in Y$ and send it via \eqref{eq:bgisos} to an element $b_y \in B(S)$. The restriction of $b_y$ to $\mb{D}$ is the element of $\tx{Hom}(\mb{D},S)^\Gamma=[Y \otimes \Q]^\Gamma$ given by $N^\diamond(y)$ according to Diagram \eqref{eq:bgdiag}. The image of this element under $\gamma$ is then $-N^\diamond(y) \in Y \otimes \Q/Y$. On the other hand, let $c_y \in H^1(u \rw W,S \rw S)$ be the image of $b_y$ under \eqref{eq:compmap}. Then $c_y|_{u} = \phi \circ b_y|_{\mb{D}}$. To describe this, we use Lemma \ref{lem:phiexp}. It tells us that the dual of $\phi_n$ is the map $\Z/n\Z \rw (\Z/[E:F]\Z)[\Gamma_{E/F}]$ that sends $1 \in \Z/n\Z$ to $-\frac{[E:F]}{n}\sum_{\sigma \in \Gamma_{E/F}}[\sigma]$. If we identify $X^*(u)=\Q[\Gamma]$ and $X^*(\mb{D})=\Q$, this means that the dual of $\phi$ is the composition
\[ \Q \rw \Q \rw \Q/\Z \rw \Q/\Z[\Gamma] \]
of the negation, the natural projection, and the diagonal embedding. Hence composing $b_y|_{\mb{D}}$ with $\phi$ is the same as sending $N^\diamond(y)$ under
\[ [Y \otimes \Q]^\Gamma = \tx{Hom}(X,\Q)^\Gamma \rw \tx{Hom}(X,\Q/\Z[\Gamma])^\Gamma = \frac{Y \otimes \Q}{Y} \]
the result of which is $-N^\diamond(y)$.

We have thus proved the commutativity of the above diagram for both $\alpha$ and $\beta$. This proves the lemma for all tori $S$ for which $Y/IY[\tx{tor}]=0$. In particular, the lemma is proved for induced tori. The general case can be easily reduced to the case of induced tori. Indeed, let $S$ be any torus and let $E/F$ be a finite extension splitting $S$. Then $Y$ is a finitely generated $\Z[\Gamma_{E/F}]$-module and we choose a free $\Z[\Gamma_{E/F}]$-module $\tilde Y$ with a surjection $\tilde Y \rw Y$. If $\tilde S$ is the torus with $X_*(\tilde S)=\tilde Y$ we obtain a surjection of tori $\tilde S \rw S$. According to \cite[Proposition 10.4]{KotBG} the natural map $B(\tilde S) \rw B(S)$ is surjective. This, together with the equality $\alpha_{\tilde S}=\beta_{\tilde S}$ that we have just shown implies $\alpha_S=\beta_S$.
\end{proof}

Let now $G$ be a connected reductive group defined over $F$. In order to discuss how the comparison map \eqref{eq:compmap} translates under the dualities \eqref{eq:bgdualg} and \eqref{eq:ridualg} we need a convenient presentation of the cover $\hat{\bar G}$. For this, let $Z_n \subset Z(G)$ be the preimage in $Z(G)$ of $(Z(G)/Z(G_\tx{der}))[n]$. Then the $Z_n$ form an exhaustive tower of finite subgroups of $Z(G)$. Set $G_n = G/Z_n$, then $G_n = G_\tx{ad} \times Z(G_n)$ with $Z(G_n) = Z(G_1)/Z(G_1)[n]$, and $Z(G_1)=Z(G)/Z(G_\tx{der})$. Note that $Z(G_1)$, and hence also $Z(G_n)$, is a torus. Dually we obtain
\[ \hat G_n = \hat G_\tx{sc} \times \hat C_n \]
where $\hat C_n$ is the torus dual to $Z(G_n)$. It will be convenient to identify $\hat C_n=\hat C_1=Z(\hat G)^\circ$. Then the map $\hat C_m \rw \hat C_n$ becomes the $m/n$-power map on $\hat C_1$. We obtain
\[ \hat{\bar G} = \varprojlim \hat G_n = \hat G_\tx{sc} \times \hat C_\infty, \quad \hat C_\infty = \varprojlim \hat C_n. \]
Elements of $Z(\hat{\bar G})$ can thus be written as tuples $(a,(b_n)_n)$, where $a \in Z(\hat G_\tx{sc})$ and $b_n \in \hat C_1$ with $b_m^\frac{m}{n}=b_n$ for $n|m$. We make explicit the condition of $(a,(b_n)_n)$ to belong to each of the subgroups $Z(\hat{\bar G})^{+,\circ}=Z(\hat{\bar G})^{\Gamma,\circ} \subset Z(\hat{\bar G})^\Gamma \subset Z(\hat{\bar G})^+$ as follows. To be in $Z(\hat{\bar G})^+$, a tuple $(a,(b_n)_n)$ must have a $\Gamma$-fixed image in $Z(\hat G)$. This image is simply $a_\tx{der} \cdot b_1$, where $a_\tx{der}$ is the image of $a$ in $Z(\hat G_\tx{der})$. The condition of belonging to $Z(\hat{\bar G})^\Gamma$ is $a \in Z(\hat G_\tx{sc})^\Gamma$ and $b_n \in \hat C_1^\Gamma$. The condition of belonging to $Z(\hat{\bar G})^{+,\circ}=Z(\hat{\bar G})^{\Gamma,\circ}$ is $a=1$ and $b_n \in \hat C_1^{\Gamma,\circ}$.

\begin{pro} \label{pro:compg}
	Under the dualities \eqref{eq:bgdualg} and \eqref{eq:ridualg}, the comparison map \eqref{eq:compmap} is translated to the map
	\begin{equation} \label{eq:mapcompz} \pi_0((Z(G_\tx{sc}) \times \hat C_\infty)^+) \rw Z(\hat G)^\Gamma \end{equation}
	sending a tuple $(a,(b_n)_n)$ with $a \in Z(G_\tx{sc})$ and $b_n \in \hat C_n$ to
	\[ \frac{a_\tx{der} \cdot b_1}{N_{E/F}(b_{[E:F]})} \]
	for a sufficiently large finite Galois extension $E/F$.
\end{pro}

Before we give the proof, let us note that this map is well-defined. By assumption, $a_\tx{der} \cdot b_1 \in Z(\hat G)^\Gamma$, and moreover $N_{E/F}(b_{[E:F]}) \in Z(\hat G)^\Gamma$, so the image of this map does belong to $Z(\hat G)^\Gamma$. The term $N_{E/F}(b_{[E:F]})$ is independent of the choice of $E/F$, and is well-defined provided $\Gamma_E$ acts trivially on $\hat C_1$. Finally, if $(a,(b_n)_n) \in Z(\hat{\bar G})^{+,\circ}$, then $a=1$ and $b_n \in \hat C_1^{\Gamma,\circ}$. Therefore $N_{E/F}(b_{[E:F]})=b_{[E:F]}^{[E:F]}=b_1$, so the image of $(a,(b_n)_n)$ is indeed equal to $1$.

\begin{proof}
Let $S \subset G$ be an elliptic maximal torus. Then for each $b \in B(S)$ the restriction $b|_\mb{D}$ takes values in $Z(G)$, because $\mb{D}$ is a split pro-torus. It follows that \eqref{eq:compmap} maps $B(S)$ to the subgroup $H^1(u \rw W,Z(G) \rw S)$ of $H^1(u \rw W,S \rw S)$. We can write this subgroup as the colimit
\[ H^1(u \rw W,Z(G) \rw S) = \varinjlim_n H^1(u \rw W,Z_n \rw S) \]
where $Z_n \subset Z(G)$ is as above. We can describe the subgroup of $\frac{Y \otimes \Q}{IY}[\tx{tor}]$ to which $H^1(u \rw W,Z(G) \rw S)$ corresponds under the isomorphism \eqref{eq:riisos} as follows. The quotient $S_n=S/Z_n$ is an elliptic maximal torus of $G_n$ and equals $S_\tx{ad} \times Z(G_n)$. Thus $X_*(S_n)= X_*(S_\tx{ad}) \oplus \frac{1}{n}X_*(Z(G_1))$ and if we let $\bar S=\varinjlim S_n$, then we get
\[ \bar Y = X_*(\bar S) = X_*(S_\tx{ad}) \oplus X_*(Z(G_1)) \otimes \Q \subset Y \otimes \Q. \]
Now let $y \in Y$ and consider the element $y-N^\diamond(y) \in Y\otimes \Q$ that is the image of $y$ under the map of Proposition \ref{pro:comps}. Since $S$ is elliptic, we have $N^\diamond(y) \in [X_*(S) \otimes \Q]^\Gamma = [X_*(Z(G))\otimes\Q]^\Gamma$. Under the natural pairing between $X^*(S)$ and $X_*(S)$ the element $N^\diamond(y)$ thus annihilates $X^*(S_\tx{ad}) \otimes \Q$, which implies that its image in $X_*(S_\tx{ad})\otimes \Q$ is zero. It follows that in the decomposition $Y \otimes \Q = X_*(S_\tx{ad}) \otimes \Q \oplus X_*(Z(G))\otimes\Q$ the element $y-N^\diamond(y) \in Y \otimes \Q$ has the coordinates $(y,y-N^\diamond(y))$. The map
\[ Y \rw \bar Y,\qquad y \mapsto (y,y-N^\diamond(y)) \]
dualizes to the map
\[ \hat S_\tx{sc} \times \hat C_\infty \rw \hat S,\qquad (a,(b_n)) \mapsto \frac{a_\tx{der}\cdot b_1}{N_{E/F}(b_{[E:F]})}. \]
To complete the proof, we use \cite[Proposition 5.3]{Kot85}, which says that the image of $B(S)$ in $B(G)$ equals $B(G)_\tx{bas}$, together with the fact that the map $B(S) \rw B(G)_\tx{bas}$ dualizes under \eqref{eq:bgdualg} to the inclusion $Z(\hat G)^\Gamma \rw \hat S^\Gamma$, while the map $H^1(u \rw W,Z(G) \rw S) \rw H^1(u \rw W,Z(G) \rw G)$ dualizes under \eqref{eq:ridualg} to the inclusion $\pi_0(Z(\hat{\bar G})^+) \rw \pi_0([\hat{\bar S}]^+)$.
\end{proof}

\section{The relationship between $\tx{LLC}_\tx{rig}$ and $\tx{LLC}_\tx{iso}$} \label{sec:compllc}

Let $G$ be a quasi-split connected reductive group defined over $F$. Let $\psi : G \rw G'$ be an inner twist. In this section we are going to compare two different statements of the refined local Langlands correspondence for $G'$. One is based on Kottwitz's cohomology set $B(G)$ of isocrystals with $G$-structure and is formulated in \cite[\S2.4]{KalIso}; we shall refer to it as $\tx{LLC}_\tx{iso}$. The other one is based on the cohomology set $H^1(u \rw W,Z(G) \rw G)$ and is formulated in \cite[\S5.4]{KalRI}; we will call it $\tx{LLC}_\tx{rig}$.

The statement $\tx{LLC}_\tx{iso}$ is defined for the given inner twist $\psi$ if and only if the class in $H^1(\Gamma,G_\tx{ad})$ of the 1-cocycle $\psi^{-1}\sigma(\psi)$ belongs to the image of the natural map $B(G)_\tx{bas} \rw H^1(\Gamma,G_\tx{ad})$. This map is surjective when $Z(G)$ is connected, thanks to \cite[Proposition 10.4]{KotBG}, and in this case $\tx{LLC}_\tx{iso}$ is always defined. In this section we will not assume that $Z(G)$ is connected, but instead we will assume that the class of $\psi^{-1}\sigma(\psi)$ does lift to $B(G)_\tx{bas}$. In the case when $\psi^{-1}\sigma(\psi)$ doesn't lift to $B(G)_\tx{bas}$, only the statement $\tx{LLC}_\tx{rig}$ is defined. In the next section we will establish results which allow us to compare $\tx{LLC}_\tx{rig}$ for the inner twist $\psi$ to $\tx{LLC}_\tx{iso}$ for a different group.

\subsection{Review of $\tx{LLC}_\tx{rig}$ and $\tx{LLC}_\tx{iso}$} \label{sub:reviewllc}
We will give here a brief review of the two formulations in order to establish the necessary notation. The reader my wish to consult the expository note \cite{KalSimons} as well as the references \cite{KalIso} and \cite{KalRI} for further details. The language we will use here is slightly different than in these references. This is done in order to emphasize the formal similarity of the two statements and facilitate their comparison. At the same time, we hope that the slightly different presentation given here can help to further illuminate the statements.

We also warn the reader that we have made a small change in Kottwitz's formulation of $\tx{LLC}_\tx{iso}$. The new statement is equivalent to the one given by Kottwitz when $Z(G)$ is connected, but is slightly different otherwise.

Let $\varphi : L_F \rw {^LG}$ be a tempered Langlands parameter. Set $S_\varphi = \tx{Cent}(\varphi,\hat G)$. The basic form of the local Langlands conjecture asserts the existence of an $L$-packet $\Pi_\varphi(G')$ of irreducible tempered representations of $G'(F)$. The two statements of the refined local Langlands conjecture we will review provide a parameterization of $\Pi_\varphi(G')$ and a description of its endoscopic transfer. They both depend on the choice of a Whittaker datum $\mf{w}$ for $G$ as well as on a choice of a certain 1-cocycle that lifts the 1-cocycle $\psi^{-1}\sigma(\psi) \in Z^1(\Gamma,G_\tx{ad})$.

We recall that an endoscopic datum for $G$ is a tuple $\mf{e}=(H,\mc{H},s,\xi)$ consisting of a quasi-split connected reductive group $H$, a split extension $\mc{H}$ of $\hat H$ by $W_F$ such that the homomorphism $W_F \rw \tx{Out}(\hat H)$ that it induces coincides under the canonical isomorphism $\tx{Out}(H)=\tx{Out}(\hat H)$with the homomorphism $\Gamma \rw \tx{Out}(H)$ given by the rational structure of $H$, an element $s \in Z(\hat H)^\Gamma$, and an $L$-embedding $\xi : \mc{H} \rw {^LG}$ that identifies $\hat H$ with $\tx{Cent}(\xi(s),\hat G)^\circ$.

Given a semi-simple element $s \in S_\varphi$, the pair $(s,\varphi)$ leads to an endoscopic datum $\mf{e}$ as follows. Set $\hat H=\tx{Cent}(s,\hat G)^\circ$, $\mc{H}=\hat H \cdot \varphi(W_F)$, and $\xi=\tx{id}$. The image of $\varphi$ is now trivially contained in $\mc{H}$.

We also recall that a $z$-pair for $\mf{e}$ is a tuple $\mf{z}=(H_1,\xi_1)$ consisting of a $z$-extension $H_1$ of $H$ and an $L$-embedding $\xi_1 : \mc{H} \rw {^LH_1}$ that extends the embedding $\hat H \rw \hat H_1$ dual to the projection $H_1 \rw H$. Note that if we compose $\varphi$ with $\xi_1$ we obtain a tempered Langlands parameter for $H_1$.

In what follows we will use the normalization of the transfer factor $\Delta'_\mf{w}$ described in \cite[(5.5.2)]{KS12}. It is a function that takes as arguments an element $\gamma_1 \in H_1(F)$ and an element $\delta \in G(F)$, both strongly regular semi-simple. We will also use a theorem of Steinberg which asserts that for any strongly regular semi-simple $\delta' \in G'(F)$ there exists $\delta \in G(F)$ that is stably conjugate to $\delta'$, by which we mean that the $G(\ol{F})$-conjugacy classes of $\delta$ and $\psi^{-1}(\delta')$ coincide. See \cite[Proposition 6.19]{PR94}, which is to be applied to $S'_\tx{der}=\tx{Cent}(\delta,G'_\tx{der})$.

The statement of $\tx{LLC}_\tx{iso}$ involves the choice of a 1-cocycle $x_\tx{iso} : \mc{E} \rw G$ such that the image of $x_\tx{iso}$ in $Z^1(\Gamma,G_\tx{ad})$ is equal to the 1-cocycle $\sigma \mapsto \psi^{-1}\sigma(\psi)$. While such a 1-cocycle may not exist in general, we are operating in this section under the assumption that it does. The pair $(\psi,x_\tx{iso})$ is then called an extended pure inner twist. The duality \eqref{eq:bgdualg} turns the cohomology class $[x_\tx{iso}]$ into a character $\<[x_\tx{iso}],-\>$ of $Z(\hat G)^\Gamma$.

Let $\delta \in G(F)$ and $\delta' \in G'(F)$ be strongly regular semi-simple elements and assume that they are stably conjugate. For any $g \in G(\ol{F})$ with $\delta'=\psi(g\delta g^{-1})$ the 1-cocycle
\[ \mc{E} \rw G,\qquad e \mapsto g^{-1}x_\tx{iso}(e)\sigma_e(g) \]
takes values in $S=\tx{Cent}(\delta,G)$. Its class is independent of the choice of $g$ and will be denoted by $\tx{inv}[x_\tx{iso}](\delta,\delta') \in B(S)$. Here $\sigma_e \in \Gamma$ is the image of $e \in \mc{E}$ under the natural projection $\mc{E} \rw \Gamma$.

We now recall the normalization $\Delta'[\mf{w},\mf{e},\mf{z},(\psi,x_\tx{iso})]$ of the Langlands-Shelstad transfer factor for the group $G'$ from \cite[\S2.3]{KalIso}. Let $\gamma_1 \in H_1(F)$ and $\delta' \in G'(F)$ be strongly regular related elements. Write $\gamma \in H(F)$ for the image of $\gamma_1$ and $S^H=\tx{Cent}(\gamma,H)$. Choose $\delta \in G(F)$ that is stably conjugate to $\delta'$. Then
\begin{equation} \label{eq:bgtf} \Delta'[\mf{w},\mf{e},\mf{z},(\psi,x_\tx{iso})](\gamma_1,\delta')=\Delta'_\mf{w}(\gamma_1,\delta) \cdot \<\tx{inv}[x_\tx{iso}](\delta,\delta'),s_{\gamma,\delta}\>. \end{equation}
Here $s_{\gamma,\delta} \in \hat S^\Gamma$ is the image of $s \in Z(\hat H)^\Gamma$ under the composition of the natural inclusion $Z(\hat H) \rw \hat S^H$ with $\hat\phi_{\gamma,\delta}^{-1}$, where $\phi_{\gamma,\delta} : S^H \rw S$ is the unique admissible isomorphism mapping $\gamma$ to $\delta$, and $\<-,-\>$ is the duality \eqref{eq:bgduals}.

We will now formulate the statement $\tx{LLC}_\tx{iso}(\psi,x_\tx{iso})$. Let $S_\varphi \cap \hat G_\tx{sc}$ denote the subgroup of $\hat G_\tx{sc}$ consisting of elements fixed by the action of $L_F$ on $\hat G_\tx{sc}$ given by $\tx{Ad}\circ\varphi$. Let $S_\varphi^\natural$ be the quotient of $S_\varphi$ by the image in $\hat G$ of $[S_\varphi \cap \hat G_\tx{sc}]^\circ$. This is a complex algebraic group. We alert the reader that when $Z(G)$ is not connected this group is slightly different from the group $S_\varphi/[S_\varphi \cap \hat G_\tx{der}]^\circ$ originally proposed by Kottwitz. Then $\tx{LLC}_\tx{iso}(\psi,x_\tx{iso})$ asserts that there is a bijection between the $L$-packet $\Pi_\varphi(G')$ and the set $\tx{Irr}(S_\varphi^\natural,[x_\tx{iso}])$ of those irreducible algebraic representations of $S_\varphi^\natural$ whose restriction of $Z(\hat G)^\Gamma$ is $\<[x_\tx{iso}],-\>$-isotypic. If for $\pi \in \Pi_\varphi(G')$ we denote by $\<\pi,-\>$ the character of the corresponding irreducible representation of $S_\varphi^\natural$, and by $\Theta_\pi$ the Harish-Chandra character of the representation $\pi$, then for any semi-simple element $s \in S_\varphi$ we can form the virtual character
\begin{equation} \label{eq:scharbg}
\Theta_{\varphi,[x_\tx{iso}]}^s = e(G')\sum_{\pi \in \Pi_\varphi(G')} \<\pi,s\>\Theta_\pi.
\end{equation}
Here $e(G')$ is the Kottwitz sign \cite{Kot83} of $G'$.
Then $\tx{LLC}_\tx{iso}(\psi,x_\tx{iso})$ asserts further that for all $f' \in \mc{C}^\infty_c(G'(F))$ the following character identity should hold
\begin{equation} \label{eq:charidbg}
\Theta^1_{\xi_1 \circ \varphi,1}(f^H) = \Theta^s_{\varphi,[x_\tx{iso}]}(f').
\end{equation}
Here we have constructed an endoscopic datum $\mf{e}$ from $s$ and $\varphi$ and have chosen an arbitrary $z$-pair $\mf{z}$ for $\mf{e}$. The function $f^H \in \mc{C}^\infty_c(H_1(F))$ is chosen to have matching orbital integrals with $f'$ with respect to the transfer factor $\Delta'[\mf{w},\mf{e},\mf{z},(\psi,x_\tx{iso})]$ as defined in \cite[\S5.5]{KS99}.

We will now review $\tx{LLC}_\tx{rig}$. It involves the choice of $x_\tx{rig} \in Z^1(u \rw W,Z(G) \rw G)$ whose image in $Z^1(\Gamma,G_\tx{ad})$ is equal to $\sigma \mapsto \psi^{-1}\sigma(\psi)$. The existence of this $x_\tx{rig}$ is guaranteed by \cite[Corollary 3.8]{KalRI}. The pair $(\psi,x_\tx{rig})$ is called a rigid inner twist. The duality \eqref{eq:ridualg} turns the cohomology class of $x_\tx{rig}$ into a character $\<[x_\tx{rig}],-\>$ of $\pi_0(Z(\hat{\bar G})^+)$. Here we are using the notation $\hat{\bar G}$ introduced in subsection \ref{sub:cohcomp}.

Let $\delta \in G(F)$ and $\delta' \in G'(F)$ be strongly regular semi-simple and assume that they are stably conjugate. For any $g \in G(\ol{F})$ with $\delta'=\psi(g\delta g^{-1})$ the 1-cocycle
\[ W \rw G,\qquad w \mapsto g^{-1}x_\tx{rig}(w)\sigma_w(g) \]
takes values in $S$. Its class is independent of the choice of $g$ and will be denoted by $\tx{inv}[x_\tx{rig}](\delta,\delta') \in H^1(u \rw W,Z(G) \rw S)$. Here $\sigma_w \in \Gamma$ is the image of $w \in W$ under the natural projection $W \rw \Gamma$.

Let $\mf{e}=(H,\mc{H},s,\xi)$ and $\mf{z}=(H_1,\xi_1)$ be an endoscopic datum and $z$-pair. There is again a normalization of the transfer factor, but it involves a refinement of $\mf{e}$. This refinement is a tuple $\mf{\dot e}=(H,\mc{H},\dot s,\xi)$. The only difference is the element $\dot s \in Z(\hat{\bar H})^+$, which is a lift of $s$. Here $\hat{\bar H}$ is the inverse limit of $\hat H_n$, where the quotient $H_n=H/Z_n$ is formed by using the canonical injection $Z(G) \rw Z(H)$ to map $Z_n \subset Z(G)$ into $Z(H)$. The definition of the transfer factor is then given by
\begin{equation} \label{eq:ritf} \Delta'[\mf{w},\mf{\dot e},\mf{z},(\psi,x_\tx{rig})](\gamma_1,\delta')=\Delta'_\mf{w}(\gamma_1,\delta) \cdot \<\tx{inv}[x_\tx{rig}](\delta,\delta'),\dot s_{\gamma,\delta}\>. \end{equation}
To describe $\dot s_{\gamma,\delta}$, recall the map $Z(\hat H) \rw \hat S$ induced by the admissible isomorphism $\phi_{\gamma,\delta}$. It lifts uniquely to a map $Z(\hat{\bar H}) \rw \hat{\bar S}$ and $\dot s_{\gamma,\delta}$ is the image of $\dot s$ under this map. It is paired with $\tx{inv}[x_\tx{rig}](\delta,\delta') \in H^1(u \rw W,Z(G) \rw S)$ using the duality \eqref{eq:riduals}.

Let now $S_\varphi^+$ be the preimage of $S_\varphi$ in $\hat{\bar G}$. Then $\tx{LLC}_\tx{rig}(\psi,x_\tx{rig})$ asserts that there is a bijection between the $L$-packet $\Pi_\varphi(G')$ and the set $\tx{Irr}(\pi_0(S_\varphi^+),[x_\tx{rig}])$ of those irreducible representations of the pro-finite group $\pi_0(S_\varphi^+)$ whose restriction to $\pi_0(Z(\hat{\bar G})^+)$ is $\<[x_\tx{rig}],-\>$-isotypic. If for $\pi \in \Pi_\varphi(G')$ we denote by $\<\pi,-\>$ the character of the corresponding irreducible representation of $\pi_0(S_\varphi^+)$, then for any semi-simple element $\dot s \in S_\varphi^+$ we can form the virtual character
\begin{equation} \label{eq:scharri}
\Theta_{\varphi,[x_\tx{rig}]}^{\dot s} = e(G')\sum_{\pi \in \Pi_\varphi(G')} \<\pi,\dot s\>\Theta_\pi
\end{equation}
and $\tx{LLC}_\tx{rig}(\psi,x_\tx{rig})$ asserts further that for all $f' \in \mc{C}^\infty_c(G'(F))$ the following character identity should hold
\begin{equation} \label{eq:charidri}
\Theta^1_{\xi_1 \circ \varphi,1}(f^H) = \Theta^{\dot s}_{\varphi,[x_\tx{rig}]}(f').
\end{equation}
Here we have constructed a refined endoscopic datum $\mf{\dot e}$ from $\dot s$ and $\varphi$ and have chosen an arbitrary $z$-pair $\mf{z}$ for $\mf{e}$. The function $f^H \in \mc{C}^\infty_c(H_1(F))$ is chosen to have matching orbital integrals with $f'$ with respect to the transfer factor $\Delta'[\mf{w},\mf{\dot e},\mf{z},(\psi,x_\tx{rig})]$.

\subsection{Comparison}
We will now show that $\tx{LLC}_\tx{iso}(\psi,x_\tx{iso})$ and $\tx{LLC}_\tx{rig}(\psi,x_\tx{rig})$ are equivalent, provided $x_\tx{rig}$ is the image of $x_\tx{iso}$ under the comparison map \eqref{eq:compmap}. As in the previous subsection we fix a Whittaker datum $\mf{w}$ for $G$ and let $\varphi : L_F \rw {^LG}$ be a tempered Langlands parameter. Recall from subsection \ref{sub:cohcomp} that $\hat{\bar G} = \hat G_\tx{sc} \times \hat C_\infty$. We can thus write elements of $S_\varphi^+ \subset \hat{\bar G}$ as pairs $(a,(b_n)_n)$ with $a \in \hat G_\tx{sc}$ and $b_n \in \hat C_n$. Taking our cue from Proposition \ref{pro:compg} we introduce the homomorphism
\begin{equation} \label{eq:mapcomps} S_\varphi^+ \rw S_\varphi,\quad (a_\tx{der},(b_n)) \mapsto \frac{a \cdot b_1}{N_{E/F}(b_{[E:F]})}. \end{equation}
Here again $a_\tx{der} \in \hat G_\tx{der}$ is the image of $a$, $E/F$ is a suitably large finite Galois extension and the expression $N_{E/F}(b_{[E:F]})$ is independent of the choice of $E/F$. The fact that this map is a group homomorphism is clear since the elements $b_n$ are central. Furthermore, we have $S_\varphi^{+,\circ} = \tx{Cent}(\varphi,\hat{\bar G})^\circ = [S_\varphi \cap \hat G_\tx{sc}]^\circ \times \hat C_\infty^\circ$. Thus for $(a,(b_n)_n) \in S_\varphi^{+,\circ}$ we have $N_{[E:F]}(b_{[E:F]})=b_1$ and the image of $(a,(b_n)_n)$ in $S_\varphi$ is simply $a \in [S_\varphi \cap \hat G_\tx{sc}]^\circ$, showing that \eqref{eq:mapcomps} induces a group homomorphism $\pi_0(S_\varphi^+) \rw S_\varphi^\natural$.

Let $x_\tx{iso} : \mc{E} \rw G$ be a 1-cocycle whose image in $Z^1(\Gamma,G_\tx{ad})$ equals $\psi^{-1}\sigma(\psi)$. Let $x_\tx{rig} : W \rw G$ be the composition of $x_\tx{iso}$ with the homomorphism $W \rw \mc{E}$ of Diagram \eqref{eq:compdiag}. Thus the class $[x_\tx{rig}] \in H^1(u \rw W,Z(G) \rw G)$ of $x_\tx{rig}$ is the image of the class $[x_\tx{iso}] \in B(G)_\tx{bas}$ of $x_\tx{iso}$ under \eqref{eq:compmap}. We denote by $\<[x_\tx{iso}],-\>$ and $\<[x_\tx{rig}],-\>$ the characters of $Z(\hat G)^\Gamma$ and $\pi_0(Z(\hat{\bar G})^+)$ given by the dualities \eqref{eq:bgdualg} and \eqref{eq:ridualg}.

\begin{lem}
Pullback along \eqref{eq:mapcomps} induces a bijection
\begin{equation} \label{eq:bijirr}
\tx{Irr}(S_\varphi^\natural,[x_\tx{iso}]) \rw \tx{Irr}(\pi_0(S_\varphi^+),[x_\tx{rig}]).
\end{equation}
\end{lem}
\begin{proof}
For the proof we need to study the kernel and image of \eqref{eq:mapcomps}. By definition of $S_\varphi^+$, the map $(a,(b_n)_n) \mapsto a_\tx{der} \cdot b_1$ is a surjection onto $S_\varphi$, and we see that $S_\varphi$ is equal to the product of $Z(\hat G)^\Gamma$ with the image of \eqref{eq:mapcomps}. This implies that composing an irreducible representation of $S_\varphi^\natural$ with \eqref{eq:mapcomps} leads to an irreducible representation of $\pi_0(S_\varphi^+)$. Moreover, according to Proposition \ref{pro:compg}, if we start with an element of $\tx{Irr}(S_\varphi^\natural,[x_\tx{iso}])$ the result will be an element of $\tx{Irr}(\pi_0(S_\varphi^+),[x_\tx{rig}])$.

We have thus shown that composition with \eqref{eq:mapcomps} induces a map $\tx{Irr}(S_\varphi^\natural,[x_\tx{iso}]) \rw \tx{Irr}(\pi_0(S_\varphi^+),[x_\tx{rig}])$. We will now argue that this map is bijective. Injectivity follows immediately from the fact that $Z(\hat G)^\Gamma$ and the image of \eqref{eq:mapcomps} generate $S_\varphi^\natural$, as one sees for example by examining the characters of the irreducible representations.

For surjectivity, we study the kernel of \eqref{eq:mapcomps}. If $(a,(b_n)_n) \in S_\varphi^+$ belongs to that kernel, then $a_\tx{der}b_1N_{[E:F]}(b_{[E:F]})^{-1}$ lifts to an element $e \in [S_\varphi \cap \hat G_\tx{sc}]^\circ$. We may replace $a$ by $ae^{-1}$ without changing the class of $(a,(b_n)_n)$ modulo $S_\varphi^{+,\circ}$, thereby achieving $a_\tx{der}b_1=N_{[E:F]}(b_{[E:F]}) \in Z(\hat G)^\Gamma$. Thus $(a,(b_n)_n)$ is an element of $Z(\hat{\bar G})^+$ and moreover belongs to the kernel of \eqref{eq:mapcompz}. We conclude that under the natural map $\pi_0(Z(\hat{\bar G})^+) \rw \pi_0(S_\varphi^+)$ the kernel of \eqref{eq:mapcompz} surjects onto the kernel of \eqref{eq:mapcomps}. Since $\<[x_\tx{rig}],-\>$ is the pull-back of $\<[x_\tx{iso}],-\>$ under \eqref{eq:mapcompz}, any $\rho \in \tx{Irr}(\pi_0(S_\varphi^+),[x_\tx{rig}])$ is trivial on the kernel of \eqref{eq:mapcomps} and thus descends to a representation of the image of this map. We extend this representation to $S_\varphi^\natural$ by letting it be given by $\<[x_\tx{iso}],-\>$ on the image of $Z(\hat G)^\Gamma$ in $S_\varphi^\natural$. The result is an element of $\tx{Irr}(S_\varphi^\natural,[x_\tx{iso}])$ whose pull-back to $\pi_0(S_\varphi^+)$ equals $\rho$. This completes the proof of surjectivity.
\end{proof}

We will now compare the character identities \eqref{eq:charidbg} and \eqref{eq:charidri}. Let $\dot s_\tx{rig} \in S_\varphi^+$ and let $s_\tx{iso} \in S_\varphi$ be the image of $\dot s_\tx{rig}$ under \eqref{eq:mapcomps}. By construction of the bijection \eqref{eq:bijirr} we have
\begin{equation} \label{eq:charidrhs}
\Theta_{\varphi,[x_\tx{iso}]}^{s_\tx{iso}} = \Theta_{\varphi,[x_\tx{rig}]}^{\dot s_\tx{rig}},
\end{equation}
so the right-hand-sides of \eqref{eq:charidbg} and \eqref{eq:charidri} agree. To compare the left-hand-sides, we let $\mf{\dot e}_\tx{rig}=(H,\mc{H},\dot s_\tx{rig},\xi)$ be the refined endoscopic datum corresponding to $\dot s_\tx{rig}$ and $\varphi$. The endoscopic datum corresponding to $s_\tx{iso}$ and $\varphi$ is then $\mf{e}_\tx{iso}=(H,\mc{H},s_\tx{iso},\xi)$. That is, the terms $H$, $\mc{H}$, and $\xi$
 are common to both $\mf{\dot e}_\tx{rig}$ and $\mf{e}_\tx{iso}$. The reason for this is that if we write $\dot s_\tx{rig} = (a,(b_n)_n)$, the image of $\dot s_\tx{rig}$ in $S_\varphi$ under the natural projection $S_\varphi^+ \rw S_\varphi$ is equal to $a_\tx{der} \cdot b_1$ and differs from $s_\tx{iso}$ only by the element $N_{E/F}(b_{[E:F]}) \in Z(\hat G)^\Gamma$. In particular, we may fix a $z$-pair $\mf{z}=(H_1,\xi_1)$ that serves both $\mf{\dot e}_\tx{rig}$ and $\mf{e}_\tx{iso}$.

We claim that for any strongly regular semi-simple elements $\gamma_1 \in H_1(F)$ and $\delta' \in G'(F)$ we have
\[ \Delta'[\mf{w},\mf{\dot e}_\tx{rig},\mf{z},(\psi,x_\tx{rig})](\gamma_1,\delta') = \Delta'[\mf{w},\mf{e}_\tx{iso},\mf{z},(\psi,x_\tx{iso})](\gamma_1,\delta'). \]
For this, fix $\delta \in G(F)$ and $g \in G(\ol{F})$ such that $\delta'=\psi(g\delta g^{-1})$. Setting as before $S=\tx{Cent}(\delta,G)$ we have $\tx{inv}[x_\tx{rig}](\delta,\delta') \in H^1(u \rw W,Z(G) \rw S)$ represented by the 1-cocycle $w \mapsto g^{-1}x_\tx{rig}(w)\sigma_w(g)$ as well as $\tx{inv}[x_\tx{iso}](\delta,\delta') \in B(S)$ represented by the 1-cocycle $e \mapsto g^{-1}x_\tx{iso}(e)\sigma_e(g)$. Since $x_\tx{iso}$ is the composition of $x_\tx{rig}$ with the homomorphism $W \rw \mc{E}$ of Diagram \eqref{eq:compdiag}, the same is true for the 1-cocycles representing the two invariants. In other words, $\tx{inv}[x_\tx{iso}](\delta,\delta')$ is the image of $\tx{inv}[x_\tx{rig}](\delta,\delta')$ under the map \eqref{eq:compmap} for the torus $S$. Proposition \ref{pro:compg} applied to the torus $S$ then implies that
\[ \<\tx{inv}[x_\tx{iso}](\delta,\delta'),s_\tx{iso}\> = \<\tx{inv}[x_\tx{rig}](\delta,\delta'),\dot s_\tx{rig}\>\]
and this proves the claim about the equality of transfer factors. This in turn implies that the function $f^H$ occurring in the left-hand-side of \eqref{eq:charidbg} is the same as the function $f^H$ occurring in the left-hand-side of \eqref{eq:charidri}. Thus the two left-hand-sides are equal. This shows that the equations \eqref{eq:charidbg} and \eqref{eq:charidri} are equivalent.

\section{Reducing $\tx{LLC}_\tx{rig}$ to the case of groups with connected center} \label{sec:redcent}

In the last section we showed that when $G$ is a connected reductive group defined and quasi-split over $F$ and $\psi : G \rw G'$ is an inner twist whose corresponding class in $H^1(\Gamma,G_\tx{ad})$ lifts to an element $[x_\tx{iso}] \in B(G)_\tx{bas}$, then $\tx{LLC}_\tx{iso}(\psi,x_\tx{iso})$ is equivalent to $\tx{LLC}_\tx{rig}(\psi,x_\tx{rig})$, where $x_\tx{rig}$ is the image of $x_\tx{iso}$ under the comparison map \eqref{eq:compmap}. When $G$ does not have connected center, then the class of $\psi$ may fail to lift to $B(G)_\tx{bas}$ (for example, this is always the case when $G$ is simply connected). In that case we do not have a statement for $\tx{LLC}_\tx{iso}$. There is however a statement for $\tx{LLC}_\tx{rig}$, since the class of $\psi$ always lifts to $H^1(u \rw W,Z(G) \rw G)$. In this section we will construct for any connected reductive group $G$ an embedding $G \rw G_z$ into a connected reductive group $G_z$ that has connected center and comparable endoscopy. We will also construct an inner twist $\psi_z : G_z \rw G'_z$ corresponding to $\psi$. We will then show that $\tx{LLC}_\tx{rig}(\psi,x_\tx{rig})$ is equivalent to $\tx{LLC}_\tx{rig}(\psi_z,x_\tx{rig})$, for any $x_\tx{rig} \in Z^1(u \rw W,Z(G) \rw G)$ lifting $\psi^{-1}\sigma(\psi)$. Combining this with the result of the previous section, and using the fact that now there does exist $[x_\tx{iso}] \in B(G_z)_\tx{bas}$ lifting the class of $\psi_z$, this implies that $\tx{LLC}_\tx{rig}(\psi,x_\tx{rig})$ is equivalent to $\tx{LLC}_\tx{iso}(\psi_z,x_\tx{iso})$ provided that the images of $[x_\tx{rig}] \in H^1(u \rw W,Z(G) \rw G)$ and $[x_\tx{iso}] \in B(G_z)_\tx{bas}$ in $H^1(u \rw W,Z(G_z) \rw G_z)$ coincide. In other words, once the validity of $\tx{LLC}_\tx{iso}$ is established for all extended pure inner twists of connected reductive groups with connected center, it implies the validity of $\tx{LLC}_\tx{rig}(\psi,x_\tx{rig})$ for all inner twists $\psi : G \rw G'$ of connected reductive groups, without assumption on the center, and some suitable $x_\tx{rig}$ lifting $\psi^{-1}\sigma(\psi)$. The final step would then be to establish the validity of $\tx{LLC}_\tx{rig}(\psi,x_\tx{rig})$ for all $x_\tx{rig}$ lifting $\psi^{-1}\sigma(\psi)$, not just those corresponding to elements $[x_\tx{iso}] \in B(G_z)_\tx{bas}$. This will be addressed in the next section.

\subsection{$z$-embeddings} \label{sub:zemb}
We will introduce here the notion of a $z$-embedding and collect some of its properties. A $z$-embedding is a procedure which embeds a given connected reductive group $G$ over a $p$-adic field $F$ into a connected reductive group $G_z$ with comparable endoscopy and connected center. The idea about the construction of $G_z$ is due to Kottwitz, who communicated it verbally to the author some years ago. It forms the core of Proposition \ref{pro:zembex}. It turns out however, that the procedure of taking a $z$-embedding is not directly compatible with passage to endoscopic groups, and it is also not transitive. Luckily, a somewhat weaker notion, that of a pseudo-$z$-embedding, does have the necessary flexibility. For this reason we work in this subsection with the weaker notion, which turns out to suffice for our applications. The main properties of pseudo-$z$-embeddings are the fact that their representation theory and endoscopy is related to that of the original group in a very close and straightforward way, and that they are in some sense stable under taking endoscopic groups and also under iteration.

We alert the reader that the requirement that a (pseudo-)$z$-embedding have the same endoscopy as the original group makes it a much more delicate object than an arbitrary group with connected center into which the original group embeds. For example, the embedding of $\tx{SL}_n$ into $\tx{GL}_n$ is not a $z$-embedding. Furthermore, a $z$-embedding is not simply the dual notion to a $z$-extension, as it has to satisfy a more stringent cohomological requirement. Finally, we want to point out that a $z$-embedding is always a ramified group, even if the original group is unramified. This additional ramification is benign, as it only affects the center, but it is nonetheless present. This makes the application of this notion to a global setting problematic. Thankfully, our needs here are purely local.

\subsubsection{Definition and construction}

Let $G$ be a connected reductive group defined over $F$.
\begin{dfn} A pseudo-$z$-embedding of $G$ is an embedding $G \rw G_z$ of $G$ into a connected reductive group $G_z$ defined over $F$, subject to the following conditions
\begin{enumerate}
\item $G_z/G$ is a torus;
\item $H^1(F,G_z/G)=1$;
\item the natural map
$H^1(F,Z(G)) \rightarrow H^1(F,Z(G_z))$
is bijective.
\end{enumerate}
If moreover $Z(G_z)$ is connected and $G_z/G$ is an induced torus we will call this a $z$-embedding.
\end{dfn}

\begin{pro} \label{pro:zembex}
Let $Z$ be a diagonalizable group defined over $F$. There exists an embedding $Z \rw T$ of $Z$ into a torus $T$ defined over $F$ with the property that $T/Z$ is an induced torus and $H^1(\Gamma,Z) \rw H^1(\Gamma,T)$ is a bijection.
\end{pro}
\begin{proof}
Let $Z \rightarrow T_0$ be an embedding of $Z$ into an arbitrary $F$-torus $T_0$ and let $C_0$ be the cokernel that embedding. Let $K_1/F$ be the splitting extension of $T_0$, and let $K/K_1$ be an extension which we will specify in a moment. Put $C=\tx{Res}_{K/F}(C_0 \times K)$. Since $C_0 \times K$ is split, $C$ is
induced. Let $T$ be the fiber product of $T_0$ and $C$ over $C_0$. This is a diagonalizable group and a quick look at its character module reveals that it is in fact a torus. We obtain the diagram
\[ \xymatrix{
1\ar[r]&Z\ar[r]\ar@{=}[d]&T\ar[r]\ar[d]&C\ar[r]\ar[d]^{N_{K/F}}&1\\
1\ar[r]&Z\ar[r]&T_0\ar[r]&C_0\ar[r]&1\\
} \]
where $N_{K/F}$ is the norm map. Since $C$ is induced, $H^1(F,C)$ vanishes and hence the natural map $H^1(F,Z) \rightarrow H^1(F,T)$ is surjective (for any
choice of $K$). We claim that we can choose $K$ in such a way that this map is also injective. This is equivalent to demanding that the map
\[ C(F)\stackrel{N_{K/F}}{\lrw} C_0(F) \lrw H^1(F,Z) \]
be trivial. We split this map as follows
\[ C(F) = C_0(K) \stackrel{N_{K/K_1}}{\lrw} C_0(K_1) \stackrel{N_{K_1/F}}{\lrw} C_0(F) \rw H^1(F,Z). \]

Fix an isomorphism $[\mb{G}_{m,K_1}]^n \rw C_0 \times K_1$. Then we have
\[ \xymatrix{
C_0(K)\ar[r]^{N_{K/K_1}}&C_0(K_1)\ar[r]^{N_{K_1/F}}&C_0(F)\ar[r]&H^1(F,Z)\\
[K^\times]^n\ar[r]^{[N_{K/K_1}]^n}\ar[u]&[K_1^\times]^n\ar[u]
} \]
If $i_k$ denotes the inclusion of the $k$-th coordinate, then the map
\[ K_1^\times \stackrel{i_k}{\lrw} [K_1^\times]^n \rw H^1(F,Z) \]
is continuous and its target is finite, so its kernel is a norm subgroup of $K_1^\times$. The intersection of these norm subgroups for $1 \leq k \leq n$ is
again a norm subgroup, and we choose $K$ to be the corresponding abelian extension of $K_1$.

We have thus shown the existence of an extension $K/K_1$ for which the canonical map $H^1(F,Z) \rw H^1(F,T)$ is bijective and this completes the proof of the proposition. For the next corollary it will be useful to know that there is a natural choice for the extension $K$ once $T_0$ has been fixed. Let $\Theta$ be the preimage in $C_0(K_1)$ of the kernel of $C_0(F) \rw H^1(F,Z)$. For any two extensions $K,K'$ of $K_1$ we have $\tx{im}(N_{K \cap K'/K_1})=\tx{im}(N_{K/K_1}) \cdot \tx{im}(N_{K'/K_1})$. Thus the set of extensions $K/K_1$ for which $\tx{im}(N_{K/K_1}) \subset \Theta$ has a smallest element, namely their intersection.

\end{proof}

\begin{cor} \label{cor:zembex}
Any connected reductive $F$-group $G$ has a $z$-embedding. Moreover, if $G$ is quasi-split, there is a canonical choice for it.
\end{cor}
\begin{proof}
Apply Proposition \ref{pro:zembex} to the diagonalizable group $Z(G)$ to obtain an embedding $Z(G) \rw T$. Form the push-out
\[\xymatrix{
Z(G)\ar[r]\ar[d]&G\ar[d]\\
T\ar[r]&G_z
} \]
The maps $T \rightarrow G_z$ and $G \rightarrow G_z$ are injective, because $Z(G) \rightarrow G$ and $Z(G) \rightarrow T$ are. Moreover, the injection $T
\rightarrow G_z$ identifies $T$ with $Z(G_z)$, and
\[ \tx{coker}(G \rightarrow G_z) = \tx{coker}(Z(G) \rightarrow T) = C. \]

Assume now that $G$ is quasi-split. In the proof of Proposition \ref{pro:zembex} we made two choices -- that of the torus $T_0$ and of the field extension $K$. We already showed in that proof that there is in fact a natural choice for $K$. For $T_0$ we can now take the minimal Levi subgroup of $G$.
\end{proof}

\subsubsection{Basic properties}

Let $1 \rw G \rw G_z \rw C \rw 1$ be a pseudo-$z$-embedding.

\begin{fct} \label{fct:zin_tran} If $G_z \rw G_x$  is a pseudo-$z$-embedding, then so is $G \rw G_x$.
\end{fct}
\begin{proof}
The map $H^1(F,Z(G)) \rw H^1(F,Z(G_x))$ is the composition of the bijections $H^1(F,Z(G)) \rw H^1(F,Z(G_z))$ and $H^1(F,Z(G_z)) \rw H^1(F,Z(G_x))$ and thus itself bijective. Moreover, the reductive group $G_x/G$ is an extension of the torus $G_x/G_z$ by the torus $C$ and hence itself a torus with $H^1(F,G_x/G)=1$.
\end{proof}

\begin{fct} \label{fct:surj} The map $Z(G_z)(F) \rightarrow C(F)$ is surjective and $G_z(F) = Z(G_z)(F) \cdot G(F)$.
\end{fct}
\begin{proof} The injectivity of $H^1(F,Z(G)) \rightarrow H^1(F,Z(G_z))$ implies the first point. For the second, we note that $G_{z,\tx{der}} \subset G$ and hence we have an exact sequence
\[ 1 \rightarrow Z(G) \rightarrow Z(G_z) \times G \rightarrow G_z \rightarrow 1. \]
The surjectivity of $Z(G_z)(F) \times G(F) \rightarrow G_z(F)$ is equivalent to the injectivity of $H^1(F,Z(G)) \rightarrow H^1(F,Z(G_z)) \times H^1(F,G)$,
which in turn follows from the injectivity of $H^1(F,Z(G)) \rightarrow H^1(F,Z(G_z))$.
\end{proof}

It follows from this fact that, if $\pi_z$ is an irreducible representation of $G_z(F)$, then its restriction $\pi$ to $G(F)$ is still irreducible. Conversely any irreducible representation
$\pi$ of $G(F)$ can be extended to an irreducible representation $\pi_z$ of $G_z(F)$ -- for this one needs to choose an extension $\omega_z : Z(G_z)(F)
\rightarrow \C^\times$ of the central character $\omega$ of $\pi$. Then $\pi \otimes \omega_z$ is a representation of $G(F) \times Z(G_z)(F)$ which factors
through the surjection $G(F) \times Z(G_z)(F) \rightarrow G_z(F)$. The set of extensions of $\pi$ to a representation of $G_z(F)$ is a torsor
for $C(F)^D$.

\begin{fct} \label{fct:zin_filt}
If $G \rw G_1$ and $G \rw G_2$ are pseudo-$z$-embeddings, then there exists a connected reductive group $G_3$ with embeddings $G_1 \rw G_3$ and $G_2 \rw G_3$ that are both pseudo-$z$-embeddings.
\end{fct}
\begin{proof}
We construct $G_3$ as the push-out $G_1 \times_{Z(G)} Z(G_2)$, i.e. the quotient of $G_1 \times Z(G_2)$ by the subgroup $\{ (z,z^{-1})| z \in Z(G)\}$. The map $G_1 \rw G_3$ given by $g \mapsto (g,1)$ is injective and its cokernel is $Z(G_2)/Z(G)=G_2/G$. It is a pseudo-$z$-embedding, because $H^1(F,G_2/G)=1$ and $Z(G_2)(F) \rw [G_2/G](F)$ is surjective. The map $G_2 \rw G_3$ given by writing an element $g_2 \in G_2$ as a product $g \cdot z_2$ with $g \in G$ and $z_2 \in Z(G_2)$ and mapping it to $(g,z_2) \in G_3$ is well-defined and injective, and its cokernel is $G_1/G$. This map is also a pseudo-$z$-embedding because $H^1(F,G_1/G)=1$ and $Z(G_1)(F) \rw [G_1/G](F)$ is surjective.
\end{proof}

\begin{fct} \label{fct:zin_tor}
We have mutually inverse bijections
\begin{eqnarray*}
\left\{\textrm{maximal tori of } G\right\}&\leftrightarrow&\{\textrm{maximal tori of }G_z\}\\
T&\mapsto&Z(G_z)^\circ\cdot T\\
(T_z \cap G)^\circ&\mapsfrom&T_z
\end{eqnarray*}
\end{fct}
\begin{proof}
This follows from $G_{z,\tx{der}} \subset G$.
\end{proof}
\begin{fct} \label{fct:zin_cohbij}
Let $Z \subset G$ be a finite central subgroup. The natural map
\[ H^1(u \rw W,Z \rw G) \rw H^1(u \rw W,Z \rw G_z) \]
is bijective. If $T$ and $T_z$ correspond under the bijection of Fact \ref{fct:zin_tor}, then the natural map
\[ H^1(u \rw W,Z \rw T) \rightarrow H^1(u \rw W,Z \rw T_z) \]
is bijective.
\end{fct}
\begin{proof}
We will discuss the second map, the argument for the first being the same. From the long exact sequence for $W$-cohomology we obtain the exact sequence
\[ T_z(F) \rw C(F) \rw H^1(u \rw W,Z \rw T) \rw H^1(u \rw W,Z \rw T_z) \rw H^1(\Gamma,C). \]
The surjectivity of $H^1(u \rw W,Z \rw T) \rightarrow H^1(u \rw W,Z \rw T_1)$ follows from the vanishing of $H^1(\Gamma,C)$ and its injectivity is a consequence of the
surjectivity of $T_z(F) \rightarrow C(F)$, which follows from the surjectivity of $Z(G_z)(F) \rightarrow C(F)$ stated in Fact \ref{fct:surj}.
\end{proof}

\begin{fct} \label{fct:zin_twist}
Let $\psi : G \rightarrow G'$ be an inner twist and $u \in Z^1(F,G_\tx{ad})$ be the element s.t. $\psi^{-1}\sigma(\psi) = \tx{Ad}(u(\sigma))$. Then
there exists a connected reductive group $G_z'$ and an inner twist $\psi_z : G_z \rightarrow G_z'$ s.t. $\psi_z^{-1}\sigma(\psi_z) = \tx{Ad}(u(\sigma))$
fitting into the diagram
\[ \xymatrix{
1\ar[r]&G\ar[r]\ar[d]^{\psi}&G_z\ar[r]\ar[d]^{\psi_z}&C\ar[r]\ar@{=}[d]&1\\
1\ar[r]&G'\ar[r]&G_z'\ar[r]&C\ar[r]&1\\
} \]
\end{fct}
\pf We construct $G_z'$ as the push-out
\[ \xymatrix{
Z(G)\ar[r]\ar[d]^\psi&Z(G_z)\ar[d]\\
G'\ar[r]&G'_z
} \]
Then the map $\tx{id} \times \psi : Z(G_z) \times G \rightarrow Z(G_z) \times G'$ descends to a map $\psi_z : G_z \rightarrow G'_z$ which clearly has the
desired property.\qed

We now consider the dual side. We have the exact sequence
\begin{equation} \label{eq:dex} 1 \rw \hat C \rw \hat{G_z} \rw \hat G \rw 1. \end{equation}
Let $\Psi^+(G)$ be the set of $\hat G$-conjugacy classes of admissible $L$-homomorphisms $L_F \times \tx{SL}_2 \rw {^LG}$. This set contains all Langlands parameters (those homomorphisms that are trivial on $\tx{SL}_2$) as well as the set of Arthur parameters (those homomorphisms whose restriction to $L_F$ projects to a relatively bounded subset of $\hat G$). We will also interpret $\Psi^+(G)$ as a subset of $H^1(L_F \times \tx{SL}_2,\hat G_z)$ via the projection $^LG_z \rw \hat G_z$.

\begin{lem} \label{lem:zin_dex} The three maps
\[ H^1(W_F,\hat C)\rw H^1(W_F,Z(\hat{G_z})) \rw
H^1(W_F,\hat{G_z}) \rw\Psi^+(G_z)
\]
are injective.
\end{lem}
\begin{proof} We have
\[ \xymatrix{
H^1(W_F,\hat C)\ar[r]\ar@{=}[d]&H^1(W_F,Z(\hat{G_z}))\ar@{^(->}[d]\\
C(F)^D\ar[r]&G_z(F)^D
} \]
and the first of the three maps is injective because $G_z(F) \rw C(F)$ is surjective. The second map is injective due to \cite[1.6]{Kot84}. The third map is injective because it can be seen as the inflation map associated to the quotient map $L_F \times \tx{SL}_2 \rw W_F$.
\end{proof}

\begin{lem} \label{lem:simple} The action of $H^1(W_F,\hat C)$ on $\Psi^+(G_z)$ given by pointwise multiplication of cocycles is simple.
\end{lem}
\begin{proof}
Consider the braided crossed module $\hat G_\tx{sc} \rw \hat{G_z}$. According to \cite[\S5.3]{KalEpi}, there is a canonical isomorphism
\[ H^1(W_F,\hat G_\tx{sc} \rw \hat{G_z}) \rw Z(G_z)(F)^D. \]
Composing the restriction map $\Psi^+(G_z) \rw H^1(W_F,\hat G_z)$ with the natural map $H^1(W_F,\hat G_z) \rw H^1(W_F,\hat G_\tx{sc} \rw \hat G_z)$ we obtain the top left horizontal map in the diagram
\[ \xymatrix{
\Psi^+(G_z)\ar[r]&H^1(W_F,\hat G_\tx{sc} \rw \hat{G_z})\ar[r]^-{\cong}&Z(G_z)(F)^D\\
H^1(W_F,\hat C)\ar[r]^-{\cong}\ar[u]&H^1(W_F,1 \rw \hat C)\ar[r]^-\cong\ar[u]&C(F)^D\ar@{^(->}[u]
} \]
The simplicity of the action of $C(F)^D$ on $Z(G_z)(F)^D$ by multiplication of characters now implies the simplicity of the action of $H^1(W_F,\hat C)$ on $\Psi^+(\hat{G_z})$.
\end{proof}

\begin{lem} \label{lem:sseq} Let $\varphi_z \in \Psi^+(G_z)$ and let $\varphi \in \Psi^+(G)$ be its image. Then we have the exact sequence
\[ 1 \rw \hat C^\Gamma \rw S_{\varphi_z} \rw S_\varphi \rw 1. \]
\end{lem}
Recall that $S_\varphi = \tx{Cent}(\varphi,\hat G)$ and $S_{\varphi_z}=\tx{Cent}(\varphi_z,\hat{G_z})$.
\begin{proof}
The exact sequence \eqref{eq:dex} has an action of $W_F$, and hence also of $L_F$. We twist it by $\varphi_z$ and obtain the long exact cohomology sequence
\begin{equation} 1 \rw \hat C^\Gamma \rw S_{\varphi_z} \rw S_\varphi \rw H^1(L_F,\hat C) \rw H^1(L_F,\varphi_z,\hat{G_z}). \label{eq:pizza} \end{equation}
Here, $H^1(L_F,\varphi_z,\hat{G_z})$ is the continuous cohomology group of $\hat{G_z}$ for the action of $W_F$ given by $\varphi_z$. To prove the lemma, we must show that the map $H^1(L_F,\hat C) \rw H^1(W_F,\varphi_z,\hat{G_z})$ is injective.

By \cite[\S5.3, Prop. 35]{SerGC}, we have a bijection
\[ H^1(L_F,\varphi_z,\hat{G_z}) \rw H^1(L_F,\hat{G_z}). \]
Composing this bijection with the last arrow in the long exact cohomology sequence \eqref{eq:pizza}, we obtain a map
\[ H^1(L_F,\hat C) \rw H^1(L_F,\hat{G_z}). \]
This map is in fact the orbit map through $\varphi_z$ for the action of $H^1(L_F,\hat C)$ on $H^1(L_F,\hat{G_z})$ by multiplication of 1-cocycles. Since $\hat C$ is abelian, $H^1(L_F,\hat C)=H^1(W_F,\hat C)$. Moreover, since $\varphi_z$ belongs to the subset $\Psi^+(G_z) \subset H^1(L_F,\hat G_z)$, so does also its orbit under $H^1(W_F,\hat C)$. According to Lemma \ref{lem:simple} the orbit map is injective and the proof is complete.

\end{proof}

\begin{cor}  \label{cor:zin_liftpar} The map $\Psi^+(G_z) \rw \Psi^+(G)$ is surjective and its fibers are torsors for the action of $H^1(W_F,\hat C)$ on $\Psi^+(G_z)$
by multiplication of cocycles. This map also induces a surjection between the sets of Langlands parameters and between the sets of Arthur parameters.
\end{cor}

\begin{proof}
Let $\varphi \in \Psi^+(G)$ and let $\varphi_0 \in H^1(W_F,\hat G)$ be its restriction to $W_F$. We will argue that there exists a lift $\varphi_{0,z} \in H^1(W_F,\hat G_z)$ of $\varphi_0$. Moreover, we will show that if $\varphi_0(W_F) \subset \hat G$ is bounded, then $\varphi_{0,z}$ can be chosen in such a way that $\varphi_{0,z}(W_F) \subset \hat G_z$ is bounded.

From the exact sequence \eqref{eq:dex} of $W_F$-modules we obtain the following diagram with exact rows
\[ \xymatrix@C=1pc{
H^1(W_F,\hat C)\ar[r]\ar@{=}[d]&H^1(W_F,\hat G_z)\ar[r]\ar[d]&H^1(W_F,\hat G)\ar[r]\ar[d]&H^2(W_F,\hat C)\ar@{=}[d]\\
H^1(W_F,\hat C)\ar[r]&H^1(W_F,\hat G_\tx{sc} \rw \hat G_z)\ar[r]&H^1(W_F,\hat G_\tx{sc} \rw \hat G)\ar[r]&H^2(W_F,\hat C)\\
}\]
It implies that if $\varphi' \in H^1(W_F,\hat G_\tx{sc} \rw \hat G_z)$ is an element whose image in the group $H^1(W_F,\hat G_\tx{sc} \rw \hat G)$ is the same as the image of $\varphi_0$ there, then there exists $\varphi_{0,z} \in H^1(W_F,\hat G_z)$ mapping simultaneously to $\varphi_0$ and to $\varphi'$.

According to \cite[Proposition 4.5.1]{KalEpi} we have the functorial isomorphism $H^1(W_F,\hat G_\tx{sc} \rw \hat G) = \tx{Hom}_\tx{cts}(Z(G)(F),\C^\times)$. Note that under this isomorphism, the unitary characters of $Z(G)(F)$ correspond to precisely those elements of $H^1(W_F,\hat G_\tx{sc} \rw \hat G)$ whose image in $H^1(W_F,\tx{cok}(\hat G_\tx{sc} \rw \hat G))$ is bounded. Indeed, this image corresponds by the usual Langlands duality for tori to a character of $Z(G)^\circ(F)$ and it is bounded if and only if the character is unitary, but a character of $Z(G)(F)$ is unitary if and only if its restriction to $Z(G)^\circ(F)$ is unitary.

The image of $\varphi_0$ in $H^1(W_F,\hat G_\tx{sc} \rw \hat G)$ thus corresponds to a character $\chi_0$ of $Z(G)(F)$. Write this character as a product $\chi_u \cdot \chi_s$ where $\chi_u : Z(G)(F) \rw \C^\times$ is unitary and $\chi_s : Z(G)(F) \rw \R_{>0}$. By Pontryagin theory $\chi_u$ extends to a unitary character $\chi_{u,z} : Z(G_z)(F) \rw \C^\times$. On the other hand, $\chi_s$ kills the maximal compact subgroup $K \subset Z(G)(F)$. If $K_z \subset Z(G_z)(F)$ is the maximal compact subgroup, then $Z(G)/K \subset Z(G_z)/K_z$ is an inclusion of finite-rank free $\Z$-modules. Since $\R_{>0}$ is injective, the homomorphism $\chi_s$ extends to a homomorphism $\chi_{s,z} : Z(G_z)(F)/K_z \rw \R_{>0}$. We set $\chi_{0,z} = \chi_{u,z} \cdot \chi_{s,z}$. If $\chi_0$ is unitary, so that $\chi_s=1$, we choose $\chi_{s,z}=1$ and $\chi_{0,z}$ is unitary. Let $\varphi' \in H^1(W_F,\hat G_\tx{sc} \rw \hat G_z)$ correspond to $\chi_{0,z}$. Choose $\varphi_{0,z}$ to map to the pair $(\varphi_0,\varphi')$. Thus $\varphi_{0,z}$ lifts $\varphi_0$. Moreover, if $\varphi_0$ is bounded, then so is its image in $H^1(W_F,\tx{cok}(\hat G_\tx{sc} \rw \hat G))$, and hence the character $\chi_0$ is unitary. Then $\chi_{0,z}$ is also unitary and thus the image of $\varphi'$ in $Z^1(W_F,\tx{cok}(\hat G_\tx{sc} \rw \hat G_z))$ is bounded. To show that $\varphi_{0,z}$ is bounded consider the diagonal map
\[ \hat G_z \rw \hat G \times \tx{cok}(\hat G_\tx{sc} \rw \hat G_z).\]
The composition of $\varphi_{0,z}$ with this map is bounded, but the kernel of this map is the finite central subgroup $\hat C \cap \hat G_{z,\tx{der}}$, and thus $\varphi_{0,z}$ is itself bounded.

Our next step is to extend $\varphi_{0,z}$ to an admissible 1-cocycle $L_F \rw \hat G_z$ that lifts $\varphi|_{L_F}$. We apply Lemma \ref{lem:sseq} to $\varphi_{0,z}$ to obtain a surjective homomorphism
\[ \tx{Cent}(\varphi_{0,z}(W_F),\hat G_z) \rw \tx{Cent}(\varphi_0(W_F),\hat G) \]
of complex algebraic groups with reductive connected components. The restriction of this map to the neutral connected components remains surjective and its kernel is still central. Thus this map restricts further to an isogeny on the level of derived subgroups. The restriction of $\varphi$ to $\tx{SL}_2 \subset L_F$ is a homomorphism of algebraic groups $\tx{SL}_2 \rw \tx{Cent}(\varphi_0(W_F),\hat G)^\circ_\tx{der}$ and lifts uniquely along this isogeny to a homomorphism $\tx{SL}_2 \rw \tx{Cent}(\varphi_{0,z}(W_F),\hat G_z)^\circ_\tx{der}$. We thus obtain a lift $\varphi_{1,z} : L_F \rw \hat G_z$ of $\varphi|_{L_F} \rw \hat G$. This completes the proof for the case of Langlands parameters.

To handle general elements of $\Psi^+(G)$, we repeat this argument again to accommodate the second copy of $\tx{SL}_2$. That is, we apply Lemma \ref{lem:sseq} to $\varphi_{1,z}$ and obtain an isogeny $\tx{Cent}(\varphi_{1,z}(L_F),\hat G_z)^\circ_\tx{der} \rw \tx{Cent}(\varphi(L_F),\hat G)^\circ_\tx{der}$ and obtain a lift of $\varphi|_{\tx{SL}_2} : \tx{SL}_2 \rw \tx{Cent}(\phi(L_F),\hat G)^\circ_\tx{der}$ to a homomorphism $\tx{SL}_2 \rw \tx{Cent}(\varphi_{1,z}(L_F),\hat G_z)^\circ_\tx{der}$, which together with $\varphi_{1,z}$ provides a lift $\varphi_z \in \Psi^+(G_z)$ of $\varphi$.

We have thus proved that $\Psi^+(G_z) \rw \Psi^+(G)$ is surjective and induces a surjection between the sets of Langlands parameters and between the sets of Arthur parameters. To show that the group $H^1(W_F,\hat C)$ acts transitively on the fibers of this map, consider the exact sequence \eqref{eq:dex} with action of $L_F \times \tx{SL}_2$ given by the quotient map of this group to $W_F$ and the corresponding exact sequence of pointed sets
\[ H^1(L_F \times \tx{SL}_2,\hat C) \rw \Psi^+(G_z) \rw \Psi^+(G) \]
According to \cite[\S5.7, Prop. 42]{SerGC} the group $H^1(L_F \times \tx{SL}_2,\hat C)$ acts transitively on the fibers. But since $\hat C$ is a torus this group is equal to $H^1(W_F,\hat C)$. Finally, the simplicity of this action comes from Lemma \ref{lem:simple}.
\end{proof}

\subsubsection{Endoscopy}

We continue with a pseudo-$z$-embedding $1 \rw G \rw G_z \rw C \rw 1$, but assume now that $G$, and hence also $G_z$, is quasi-split. As before we set $Z_n \subset Z(G)$ to be the preimage of $(Z(G)/Z(G_\tx{der}))[n]$ and form $G_n=G/Z_n$, $\bar G=\varinjlim G_n$, and $\hat{\bar G}=\varprojlim \hat G_n$. We also form $G_{z,n}=G_z/Z_n$, $\bar G_z = \varinjlim G_{z,n}$ and $\hat{\bar G_z}=\varprojlim \hat G_{z,n}$.

For every $n$ we have the exact sequence
\[ 1 \rw G_n \rw G_{z,n} \rw C \rw 1. \]
Since the composition $Z(G_z)(F) \rw Z(G_{z,n})(F) \rw C(F)$ is surjective, so must be the map $Z(G_{z,n})(F) \rw C(F)$. Thus the above exact sequence is a pseudo-$z$-embedding. Dually we obtain the exact sequence
\[ 1 \rw \hat C \rw \hat G_{z,n} \rw \hat G_n \rw 1 \]
and taking the limit over $n$ we arrive at the exact sequence
\[ 1 \rw \hat C \rw \hat{\bar G_z} \rw \hat{\bar G} \rw 1. \]

Let $\mf{\dot e}=(H,\mc{H},\dot s,\xi)$ be a refined endoscopic datum for $G$. We are going to construct a refined endoscopic datum $\mf{\dot e}_z = (H_z,\dot s_z,\mc{H}_z,\xi_z)$ for $G_z$. Let $\dot s_z \in \hat{\bar G_z}$ be any preimage of $\dot s$. We form $\mc{H}_z$ and $\xi_z$ using the following pull-back diagram:
\[ \xymatrix{
1\ar[r]&\hat C\ar[r]&{^LG_z}\ar[r]&^LG\ar[r]&1\\
1\ar[r]&\hat C\ar[r]\ar@{=}[u]&\mc{H}_z\ar[r]\ar[u]^{\xi_z}&\mc{H}\ar[r]\ar[u]^\xi&1\\
} \]
Dually, we construct $H_z$ as the push-out:
\[ \xymatrix{
Z(G)\ar[r]\ar[d]&H\ar[d]\\
Z(G_z)\ar[r]&H_z
} \]

\begin{lem} The quadruple $\mf{\dot e}_z$ is a refined endoscopic datum for $G_z$. Furthermore, the natural embedding $H \rw H_z$ is a pseudo-$z$-embedding with cokernel $C$.\end{lem}
\begin{proof}
Let us consider the second statement first. By construction we have the exact sequence
\[ 1 \rw H \rw H_z \rw C \rw 1. \]
It provides the exact sequence
\[ 1 \rw Z(H) \rw Z(H_z) \rw C \rw 1. \]
The surjectivity of $H^1(F,Z(H)) \rw H^1(F,Z(H_z))$ follows from the vanishing of $H^1(F,C)$, and the injectivity of the same map is equivalent to the surjectivity of $Z(H_z)(F) \rw C(F)$, which in turns follows from the fact that the composition of this map with the natural inclusion $Z(G_z)(F) \rw Z(H_z)(F)$ is equal to the surjective map $Z(G_z)(F) \rw C(F)$ of Fact \ref{fct:surj}.

To ease notation we now assume $\mc{H} \subset {^LG}$ and $\mc{H}_z \subset {^LG_z}$, so that $\xi$ and $\xi_z$ are the natural inclusions. Let $s_z \in \hat{G_z}$ be the image of $\dot s_z$. We will argue that $\mf{e}_z=(H_z,\mc{H}_z,s_z,\xi_z)$ is an endoscopic datum for $G_z$. Let $s \in \hat G$ be the image of $\dot s$.
\begin{enumerate}
\item $H_z$ is quasi-split: This holds because $H$ is quasi-split and contains $H_{z,\tx{der}}$, since the quotient $H_z/H$ is a torus.
\item $[\hat{G_z}]_{s_z}^\circ$ is a dual group for $H_z$: Let $T^H \rw T^G$ be an admissible isomorphism from a maximal torus of $H$ to a maximal
torus of $G$. We form the push-outs
\[ \xymatrix{
Z(G)\ar[r]\ar[d]&T^H\ar[d]&&Z(G)\ar[r]\ar[d]&T^G\ar[d]\\
Z(G_z)\ar[r]&T^{H_z}&&Z(G_z)\ar[r]&T^{G_z}
} \]
and obtain an isomorphism $T^{H_z} \rightarrow T^{G_z}$ from a maximal torus in $H_z$ to a maximal torus in $G_z$. Let $T^{\hat{G_z}} \subset \hat{G_z}$ be a
maximal torus containing $s_z$. Its image $T^{\hat G}$ in $\hat G$ contains $s$. There exists an
admissible isomorphism $\hat{T^{G_z}} \rw T^{\hat{G_z}}$ with the following property: The induced isomorphism $\hat{T^G} \rightarrow T^{\hat G}$ when composed
with $\hat{T^H} \rightarrow \hat{T^G}$ identifies the coroot system $R^\vee(T^H,H)$ with the root system $R(T^{\hat G},\hat H)$. But in the diagram
\[ \xymatrix{
\hat{T^{H_z}}\ar[r]\ar[d]&T^{\hat{G_z}}\ar[d]\\
\hat{T^H}\ar[r]&T^{\hat G}\\
} \]
the left arrow induces a bijection $R^\vee(T^{H_z},H_z) \rw R^\vee(T^H,H)$ while the right arrow induces a bijection $R(T^{\hat{G_z}},[\hat{G_z}]_{s_z}^\circ) \rw R(T^{\hat G},\hat G_s)$. This implies that the top horizontal arrow induces a bijection
\[R^\vee(T^{H_z},H_z) \rw R(T^{\hat{G_z}},[\hat{G_z}]_{s_z}^\circ).\]
This shows that $[\hat{G_z}]_{s_z}^\circ$ is a dual group of $H_z$.

\item $\mc{H}_z$ is an extension of $W_F$ by $\hat{H_z}$: Since ${^LG_z} \rw {^LG}$ is surjective, so is also $\mc{H}_z \rw \mc{H}$. Composing the latter map  with $\mc{H} \rw W_F$ we obtain a continuous surjective map $\mc{H}_z \rw W_F$. Its kernel is the preimage of $\hat H$ in $\mc{H}_z$. Let's call this kernel $K$ for a moment. Then we have the exact sequence of topological groups and continuous homomorphisms
\[ 1 \rw K \rw \mc{H}_z \rw W_F \rw 1. \]
We claim that this is an extension, i.e. that the induced map $\mc{H}_z/K \rw W_F$ is an isomorphism of topological groups. To prove this, we will use that $\mc{H}_z$ is locally-compact and $\sigma$-compact. Indeed, since $\mc{H}$ is a split extension of $W_F$ by $\hat H$, and both $W_F$ and $\hat H$ are locally-compact and $\sigma$-compact, so is $\mc{H}$, and so is the product ${^LG_z} \times \mc{H}$, of which $\mc{H}_z$ is a closed subgroup. It follows that $\mc{H}_z$ is locally-compact and $\sigma$-compact, and the open mapping theorem implies that the surjection $\mc{H}_z \rw W_F$ is open. This proves the claim that the natural continuous bijective homomorphism $\mc{H}_z/K \rw W_F$ is an isomorphism of topological groups.

Next, one checks that $\xi_z$ restricts to a continuous bijective homomorphism from $K$ to the preimage of $\xi(\hat H)$ in ${^LG_z}$, the latter being precisely $[\hat{G_z}]_{s_z}^\circ$. $K$ being a closed subgroup of $\mc{H}_z$ is also locally-compact and $\sigma$-compact, so $\xi_z : K \rw [\hat{G_z}]_{s_z}^\circ$ is an isomorphism of topological groups. Having already shown in the previous point that $[\hat{G_z}]_{s_z}^\circ$ is a dual group for $\hat{H_z}$, we conclude that indeed $\mc{H}_z$ is an extension of $W_F$ by $\hat{H_z}$.

\item The extension $1 \rw \hat{H_z} \rw \mc{H}_z \rw W_F \rw 1$ is split: We know that the extension $1 \rw \hat H \rw \mc{H} \rw W_F \rw 1$ is split, so we may choose a splitting $W_F \rw \mc{H}$, which we then compose with $\xi : \mc{H} \rw {^LG}$ and obtain an element $a \in Z^1(W_F,\hat G)$. According to Corollary \ref{cor:zin_liftpar}, $a$ lifts to an element $a_z \in Z^1(W_F,\hat{G_z})$, which we interpret as an $L$-homomorphism $W_F \rw {^LG_z}$. Its image is contained in the image of $\xi_z$. But $\xi_z$ is an isomorphism of topological groups onto its image, due to the local- and $\sigma$-compactness of $\mc{H}_z$, so in the end we obtain a continuous $L$-homomorphism $W_F \rw \mc{H}_z$, which is just the splitting we were looking for.

\item $\dot s_z \in Z(\hat{\bar H_z})^+$: It is enough to show $s_z \in Z(\hat H_z)^\Gamma$. Consider the exact sequence of $W_F$-modules
\[ 1 \rw \hat C \rw Z(\hat H_z) \rw Z(\hat H) \rw 1. \]
Then $s \in Z(\hat H)^\Gamma$ maps to an element of $H^1(W_F,\hat C)$ whose image in $H^1(W_F,Z(\hat H_z))$ is trivial. By Lemma \ref{lem:zin_dex} and the already proved fact that $H \rw H_z$ is a pseudo-$z$-embedding, we conclude that the image of $s$ in $H^1(W_F,\hat C)$ is trivial, so $s$ can be lifted to an element of $Z(\hat H_z)^\Gamma$. But the set of such lifts is a torsor under $\hat C^\Gamma \subset Z(\hat H_z)^\Gamma$, hence $s_z \in Z(\hat H_z)^\Gamma$ as claimed.

\item The $L$-action of $W_F$ on $\hat{H_z}$ obtained from the extension $\mc{H}_z$ is the same as the $L$-action coming from the rational structure of $H_z$: We have to show that the image of $\sigma \in \Gamma$ in  $\tx{Out}(H_z)$ corresponds via the canonical isomorphism $\tx{Out}(H_z) \cong \tx{Out}(\hat{H_z})$ to the image of $\tx{Ad}(g_\sigma)$ for some $g_\sigma \in {^LG_z}$ mapping to $\sigma$. For this it
is enough to show that the action of $\sigma$ on $R^\vee(T^{H_z},H_z)$ is translated via the isomorphism $T^{H_z} \rightarrow T^{G_z}$ to the action of an element of $w\cdot\sigma$, where $w$ belongs to the Weyl group of $T^{G_z}$. The vertical arrows in the diagram
\[ \xymatrix{
T^{H_z}\ar[r]&T^{G_z}\\
T^H\ar[r]\ar[u]&T^G\ar[u]
} \]
induce bijections of root- and coroot-systems and Weyl-groups. Since the assertion holds the bottom map, it also holds the top.
\end{enumerate}

\end{proof}

There is an inverse construction as well. Let $\mf{\dot e}_z=(H_z,\mc{H}_z,\dot s_z,\xi_z)$ be a refined endoscopic datum for $G_z$. Then $\xi_z(\hat H_z)$ contains the central torus $\hat
C$ and this gives an injection $\hat C \rightarrow \hat{H_z}$, which dually provides a surjection $H_z \rightarrow C$. Let $H = \tx{ker}(H_z \rightarrow C)$,
$\hat{H} = \tx{im}(\hat H_z \rightarrow \hat G)$, let $\dot s$ be the image of $\dot s_z$ under  $Z(\hat{\bar H_z}) \rightarrow Z(\hat{\bar H})$
and $\xi$ be the composition of $\xi_z$ with the canonical projection $\hat{G_z} \rw \hat G$. One checks that $(H,\mc{H},\dot s,\xi)$ is a refined endoscopic datum for $G$ in a way similar to the above argument.

\begin{fct} \label{fct:mib_ed} The above constructions provide mutually inverse bijections between the isomorphism classes of refined endoscopic data for $G$ and those for $G_z$. \end{fct}
\begin{proof}
The assignment $(H_z,\mc{H}_z,\dot s_z,\xi_z) \mapsto (H,\mc{H},\dot s,\xi)$ is a map between the two sets of isomorphism classes. We claim that the assignment $(H,\mc{H},\dot s,\xi) \rw (H_z,\mc{H}_z,\dot s_z,\xi_z)$ is also a map. We needed to choose a lift $\dot s_z$ of $\dot s$. The set of choices is a torsor under $\hat C^\Gamma$, which is a connected subgroup of $Z(\hat{\bar H_z})^+$ (because $1=H^1(F,C)=\pi_0(\hat C^\Gamma)^*$). The image of $\dot s_z$ in $\pi_0(Z(\hat{\bar H_z})^+)$ is thus uniquely determined by $\dot s$ and the claim is proved. Checking that the two maps are inverses of each other is straightforward.
\end{proof}

\begin{lem} Let $(H,s,\mc{H},\xi)$ and $(H_z,s_z,\mc{H}_z,\xi_z)$ correspond under the mutually inverse bijections. Then either both are elliptic or both are not. \end{lem}
 \begin{proof} By construction of $H_z$ the left square below is cocartesian, and applying the left-exact functor $X^*(-)^\Gamma$ we obtain the right square below, which
is then cartesian.
\[ \xymatrix{
Z(G)\ar[r]\ar[d]&Z(H)\ar[d]&&&X^*(Z(H_z))^\Gamma\ar[r]\ar[d]&X^*(Z(H))^\Gamma\ar[d]\\
Z(G_z)\ar[r]&Z(H_z)&&&X^*(Z(G_z))^\Gamma\ar[r]&X^*(Z(G))^\Gamma
} \]
The top map of the right diagram provides an isomorphism
\[ \tx{ker}\left(X^*(Z(H_z))^\Gamma \rw X^*(Z(G_z))^\Gamma\right) \rw \tx{ker}\left(X^*(Z(H))^\Gamma \rw X^*(Z(G))^\Gamma\right) \]
The lemma now follows from the fact that ellipticity of $H_z$ resp. $H$ is equivalent to the corresponding kernel being finite.
\end{proof}

We continue with $\mf{\dot e}_z=(H_z,\mc{H}_z,\dot s_z,\xi_z)$ and $\mf{\dot e}=(H,\mc{H},\dot s,\xi)$ corresponding under the mutually inverse bijections. Let $\mf{z}_{z}=(H_{z,1},\xi_{z,1})$ be a $z$-pair for $\mf{\dot e}_z$. Thus $H_{z,1} \rw H_z$ is a $z$-extension and $\xi_{z,1} : \mc{H}_z \rw {^LH_{z,1}}$ is an $L$-embedding extending the embedding $\hat H_z \rw \hat H_{z,1}$. From $\mf{z}_z$ we can construct a $z$-pair $\mf{z}=(H_1,\xi_1)$ for $\mf{\dot e}$ as follows: $H_1$ is the fiber product of $H_{z,1}$ and $H$ over $H_z$, and $\xi_1 : \mc{H} \rw {^LH_1}$ is the unique factoring of
\[ \mc{H}_z \stackrel{\xi_{z,1}}{\lrw} {^LH_{z,1}} \rw {^LH_1} \]
through $\mc{H}_z \rw \mc{H}$.

\begin{lem}\label{lem:mib_zp} \ \\[-20pt]
\begin{enumerate}
\item $H_1 \rw H_{z,1}$ is a pseudo-z-embedding with cokernel $C$.
\item $H_{z,1} \rw H_z$ is a z-extension with kernel $K$.
\item $(H_1,\xi_1)$ is a $z$-pair for $\mf{\dot e}$ and the map $(H_{z,1},\xi_{z,1}) \mapsto (H_1,\xi_1)$ from the set of $z$-pairs for $\mf{\dot e}_z$ to the set of $z$-pairs for $\mf{\dot e}$ has fibers which are torsors for $Z^1(W_F,\hat C)$ acting on the second component of the $z$-pairs by pointwise multiplication.
\end{enumerate}
\end{lem}
\begin{proof}
By construction we have the diagram
\[ \xymatrix{
&&1\ar[d]&1\ar[d]\\
1\ar[r]&K\ar[r]\ar@{=}[d]&H_1\ar[r]\ar[d]&H\ar[r]\ar[d]&1\\
1\ar[r]&K\ar[r]&H_{z,1}\ar[r]\ar[d]&H_z\ar[r]\ar[d]&1\\
&&C\ar@{=}[r]\ar[d]&C\ar[d]\\
&&1&1
} \]
We see that $H_1$ embeds into $H_{z,1}$ with abelian quotient, so the derived group of $H_1$ equals that of $H_{z,1}$. This shows that $H_1$ is a $z$-extension of $H$. To prove that $H_1 \rw H_{z,1}$ is a pseudo-$z$-embedding, we only need to show that $H_{z,1}(F) \rw C(F)$ is surjective. This follows from the surjectivity of $H_{z,1}(F) \rw H_z(F)$ (which relies on the fact that $K$ is induced), and the surjectivity of $H_z(F) \rw C(F)$ (as $H \rw H_z$ is a pseudo-z-embedding). The final point now follows easily from the fact that the kernels of $\mc{H}_z \rw \mc{H}$ and $^LH_{z,1} \rw {^LH_1}$ are both equal to $\hat C$.
\end{proof}

Recall that we are assuming that $G$, and hence also $G_z$, is quasi-split. The bijection of Fact \ref{fct:zin_tor} extends to a bijection between the sets of splittings of $G$ and $G_z$, as well as a bijection between the sets of Whittaker data for $G$ and $G_z$.

\begin{lem} \label{lem:zin_tf} Let $\mf{\dot e}_z=(H_z,\mc{H}_z,\dot s_z,\xi_z)$ and $\mf{\dot e}=(H,\mc{H},\dot s,\xi)$ be refined endoscopic data for $G_z$ and $G$, whose equivalence classes correspond via the bijections of Fact \ref{fct:mib_ed}. Let $\mf{z}_z=(H_{z,1},\xi_{z,1})$ and $\mf{z}=(H_1,\xi_1)$ be $z$-pairs for $\mf{\dot e}_z$ and $\mf{\dot e}$, corresponding as in Lemma \ref{lem:mib_zp}. Let $\psi : G \rightarrow G'$ and $\psi_z: G_z \rightarrow G_z'$ be
compatible inner twists, as in Fact \ref{fct:zin_twist}. Let $x \in Z^1(u \rw W,Z \rw G)$ map to $\psi^{-1}\sigma(\psi) \in Z^1(\Gamma,G_\tx{ad})$. Let $\mf{w}_z$ and $\mf{w}$ be Whittaker data that correspond to each other.

If $\gamma_1 \in H_1(F)$ and $\delta' \in G'(F)$ are strongly $G$-regular elements, then
\[ \Delta'[\mf{w}_z,\mf{\dot e}_z,\mf{z}_z,(\psi_z,x)](\gamma_1,\delta') = \Delta'[\mf{w},\mf{\dot e},\mf{z},(\psi,x)](\gamma_1,\delta'). \]
\end{lem}
\begin{proof} Choose $\delta \in G(F)$ which is stably-conjugate to $\delta'$ and let $\gamma$ be the image of $\gamma_1$ in $H(F)$. Let $T' = \tx{Cent}(\delta',G')$, $T=\tx{Cent}(\delta,G)$, $T^{H_1}=\tx{Cent}(\gamma_1,H_1)$, $T^H=\tx{Cent}(\gamma,H)$. Each of the groups $G$, $G'$, $H$, $H_1$ has the corresponding pseudo-$z$-inflation, which we denote by a subscript $z$, and each of the tori $T$, $T'$, $T^H$, $T^{H_1}$ has a torus corresponding under the bijections of Fact \ref{fct:zin_tor}, which we will also denote by a subscript $z$.

We now recall from \eqref{eq:ritf} and \cite[(5.5.2)]{KS12} that $\Delta'[\mf{w},\mf{\dot e},\mf{z},(\psi,x)](\gamma_1,\delta')$ is given by
\[ \epsilon\Delta_I^{-1}(\gamma,\delta)\Delta_{II}(\gamma,\delta)\Delta_{III_2}(\gamma_1,\delta)\Delta_{IV}(\gamma,\delta) \<\tx{inv}[x](\delta,\delta'),\dot s_{\gamma,\delta}\>. \]
Note that we are dealing with untwisted endoscopy and $\Delta_I^\tx{new}$ in \cite{KS12} is the same as the original $\Delta_I$. We will now discuss the individual terms and show that they match the corresponding terms in $\Delta'[\mf{w}_z,\mf{\dot e}_z,\mf{z}_z,(\psi_z,x)](\gamma_1,\delta')$.

Write $\mf{w}=(B_0,\chi_0)$, where $B_0 \subset G$ is a Borel subgroup and $\chi_0$ is a generic character of the $F$-points of the unipotent radical $U_0$ of $B_0$. Extend $B_0$ to a pinning $\tx{spl}=(T_0,B_0,\{X_\alpha\})$ of $G$ and choose a character $\chi_F : F \rw \C^\times$ so that $\chi_0$ corresponds to $\tx{spl}$ and $\chi_F$ as in \cite[\S5.3]{KS99}. Let $T_0^H \subset H$ be a minimal Levi subgroup. Then $\epsilon$ is the Langlands normalization of the $\epsilon$-factor $\epsilon_L(X^*(T_0)\otimes \C - X^*(T_0^H)\otimes \C,\psi_F)$. Analogously, the $\epsilon$-factor in the definition of $\Delta'[\mf{w}_z,\mf{\dot e}_z,\mf{z}_z,(\psi_z,x)](\gamma_1,\delta')$ is given by $\epsilon_L(X^*(T_{0,z})\otimes \C - X^*(T_0^{H_z})\otimes \C,\psi_F)$, where $T_{0,z} \subset G_z$ and $T_0^{H_z} \subset H_z$ correspond to $T_0$ and $T_0^H$ as in Fact \ref{fct:zin_tor}. But then $T_{0,z}$ is an extension of $C$ by $T_0$ and $T_0^{H_z}$ is an extension of $C$ by $T_0^H$. The two epsilon factors above are thus equal, due to their additivity \cite[(3.4.2)]{TateCor}.

For the discussion of the remaining factors, we fix the admissible isomorphism $T^H \rw T$ that sends $\gamma$ to $\delta$. It extends uniquely to an admissible isomorphism $T^{H_z} \rw T_z$. We furthermore choose $a$-data and $\chi$-data for $T$. Since they depend only on the roots, they work equally well for $T_z$.

The factor $\Delta_I(\gamma,\delta)$ depends on the admissible isomorphism, the splitting $\tx{spl}$ and $a$-data. Since its construction involves only the preimage of $T$ in $G_\tx{sc}$, which is the same as the preimage of $T_z$ in $G_{z,\tx{sc}}=G_\tx{sc}$, we see that this factor matches the corresponding factor for $G_z$.

The factors $\Delta_{II}$ and $\Delta_{IV}$ are also immediately seen to match their counterparts in $G_z$, because they only depend on the chosen $\chi$-data and the root-values of $\delta$.

The factor $\Delta_{III_2}(\gamma_1,\delta)$ needs closer attention. We recall briefly its construction, following loosely \cite[\S4.4]{KS99} but specializing to the non-twisted setting at hand. We have chosen $\chi$-data for $T$, which we transport via the chosen admissible isomorphism $T^H \rw T$ to obtain $\chi$-data for $T^H$. The surjection $T^{H_1} \rw T^H$ induces an $L$-embedding $^LT \rw {^LT^{H_1}}$. It also induces a bijection on the root systems, so we also obtain $\chi$-data for $T^{H_1}$. These $\chi$-data provide, according to the procedure of \cite[\S2.6]{LS87}, admissible $L$-embeddings $^LT \rw {^LG}$ and ${^LT^{H_1}} \rw {^LH_1}$. The admissible isomorphism $T^H \rw T$ induces an $L$-isomorphism $^LT \rw {^LT^H}$. We obtain the following diagram
\begin{equation} \label{dia:del32}
\xymatrix{
^LH_1&^LT^{H_1}\ar@{_(->}[l]&^LT^{H_1}\ar@{..>}[l]&^LT^H\ar@{_(->}[l]\\
\mc{H}\ar@{^(->}[u]^{\xi_1}\ar@{^(->}[r]^{\xi}&^LG&&^LT\ar@{_(->}[ll]\ar[u]_\cong
} \end{equation}
The dotted arrow is defined to be the unique $L$-automorphism of $^LT^{H_1}$ extending the identity on $\hat T^{H_1}$ and making the diagram commutative. The restriction of this $L$-automorphism to $W_F$ is then a Langlands parameter $a : W_F \rw {^LT^{H_1}}$  and
\[ \Delta_{III_2}(\gamma^{H_1},\gamma)= \<a,\gamma^{H_1}\>, \]
where $\<\cdot,\cdot\>$ is the Langlands duality pairing.
The construction of the term $\Delta_{III_2}$ contributing to $\Delta'[\mf{w}_z,\mf{\dot e}_z,\mf{z}_z,(\psi_z,x)](\gamma_1,\delta')$ is the same, but involves the analog of diagram \eqref{dia:del32} where all objects and arrows have subscript $z$. This latter diagram surjects onto \eqref{dia:del32}, with the kernel at each node being $\hat C$. In particular, we see that the composition of $a_z : W_F \rw {^LT^{H_{z,1}}}$ with the natural projection ${^LT^{H_{z,1}}} \rw {^LT^{H_1}}$ is equal to $a$. The functoriality of the pairing $\<\cdot,\cdot\>$ and the fact that $\gamma_1$ belongs to the subgroup $T^{H_1}(F)$ of $T^{H_{z,1}}(F)$ now implies that the two versions of $\Delta_{III_2}$ agree.

The final term to be compared is $\<\tx{inv}[x](\delta,\delta'),\dot s_{\gamma,\delta}\>$. Here $\tx{inv}[x](\delta,\delta') \in H^1(u \rw W,Z \rw T)$. It is a direct observation that mapping this element into $H^1(u \rw W,Z \rw T_z)$ gives the same result as mapping $\delta$ and $\delta'$ into $T_z(F)$ and $T_z'(F)$ and then computing $\tx{inv}[x](\delta,\delta')$. At the same time, tracing through the definition of $\dot s_{\gamma,\delta}$ we see that it is the image of $\dot s_{z,\gamma,\delta}$ under the projection $\hat{\bar T_z} \rw \hat{\bar T}$. The functoriality of the duality pairing $\<-,-\>$ completes the proof.
\end{proof}

\subsection{Comparison of $\tx{LLC}_\tx{rig}(\psi,x_\tx{rig})$ and $\tx{LLC}_\tx{rig}(\psi_z,x_\tx{rig})$}
We now assume that for all inner twists $\tilde\psi : \tilde G \rw \tilde G'$ of connected reductive quasi-split groups with connected center, and all $\tilde x_\tx{rig} \in Z^1(u \rw W,Z(\tilde G) \rw \tilde G)$ lifting $\tilde\psi^{-1}\sigma(\tilde\psi)$, the statement $\tx{LLC}_\tx{rig}(\tilde\psi,\tilde x)$ holds. Furthermore, we assume the following natural compatibility. If $\tilde G \rw \tilde G_z$ is a pseudo-$z$-embedding into a group $\tilde G_z$ with connected center and $\tilde\psi_z : \tilde G_z \rw \tilde G_z'$ is the inner twist compatible with $\tilde\psi$ as in Fact \ref{fct:zin_twist}, then for any tempered parameter $\tilde\varphi_z : L_F \rw {^L\tilde G_z}$ with corresponding $\tilde\varphi : L_F \rw {^L\tilde G}$ restriction of representations provides a bijection $\Pi_{\tilde\varphi_z}(\tilde G_z') \rw \Pi_{\tilde\varphi}(\tilde G')$ and this bijection is compatible with the bijection $\pi_0(S_{\tilde\varphi_z}^+) \rw \pi_0(S_{\tilde\varphi}^+)$ (see below for an argument about why the second map is bijective).

Under this assumption, we will show that $\tx{LLC}_\tx{rig}(\psi,x_\tx{rig})$ holds for any connected reductive quasi-split group $G$ with fixed Whittaker datum $\mf{w}$, inner twist $\psi : G \rw G'$, and $x_\tx{rig} \in Z^1(u \rw W,Z(G) \rw G)$ lifting $\psi^{-1}\sigma(\psi)$. For this, we choose a pseudo-$z$-embedding $G \rw G_z$ such that $G_z$ has connected center. This is possible by Corollary \ref{cor:zembex}. Let $\mf{w}_z$ be the Whittaker datum for $G_z$ determined by $\mf{w}$. Let $\psi_z : G_z \rw G_z'$ be the inner twist corresponding to $\psi$ as in Fact \ref{fct:zin_twist}.

Let $\varphi : L_F \rw {^LG}$ be a tempered Langlands parameter. Choose a tempered Langlands parameter $\varphi_z : L_F \rw {^LG_z}$ lifting $\varphi$. It exists by Corollary \ref{cor:zin_liftpar}. Let $\Pi_{\varphi_z}(G_z')$ be the corresponding tempered $L$-packet. All elements of $\Pi_{\varphi_z}(G_z')$ have the same central character and Fact \ref{fct:surj} implies that restriction to $G'(F)$ provides an injective map $\Pi_{\varphi_z}(G_z') \rw \Pi_\tx{temp}(G')$. Define $\Pi_\varphi(G')$ to be the image of this map so that we obtain a bijection
\begin{equation} \label{eq:zin_cmpp} \Pi_{\varphi_z}(G_z') \rw \Pi_\varphi(G'). \end{equation}
Applying Lemma \ref{lem:sseq} to each pseudo-$z$-embedding $G_n \rw G_{z,n}$ and taking the limit we obtain the exact sequence
\[ 1 \rw \hat C^\Gamma \rw S_{\varphi_z}^+ \rw S_\varphi^+ \rw 1. \]
Applying the right-exact functor $\pi_0$ and noting that $\hat C^\Gamma$ is connected we obtain the isomorphism
\begin{equation} \label{eq:zin_cmps} \pi_0(S_{\varphi_z}^+) \rw \pi_0(S_\varphi^+). \end{equation}
From \eqref{eq:zin_cmpp} and \eqref{eq:zin_cmps} we obtain the bijection
\begin{equation} \label{eq:zin_llc}
\tx{Irr}(\pi_0(S_\varphi^+),[x_\tx{rig}]) \rw \Pi_\varphi(G').
\end{equation}
A-priori the packet $\Pi_\varphi(G')$ and the bijection \eqref{eq:zin_llc} could depend on the choice of lift $\varphi_z$ of $\varphi$, as well as on the choice of $z$-embedding $G \rw G_z$. We claim that this is not the case. Indeed, any other lift of $\varphi$ is of the form $\varphi_z \cdot \varphi_c$ for some $\varphi_c \in Z^1(W_F,\hat C)$. Then $\pi_z \mapsto \pi_z \otimes \chi_c$ is a bijection $\Pi_{\varphi_z}(G_z') \rw \Pi_{\varphi_z \cdot \varphi_c}(G_z')$, where $\chi_c : C(F) \rw \C^\times$ is the character corresponding to $\varphi_c$. This bijection is compatible with the identity $S_{\varphi_z} = S_{\varphi_z \cdot \varphi_c}$. Since $\chi_c$ restricts trivially to $G'(F)$ we conclude that the packet $\Pi_\varphi(G')$ and the bijection \eqref{eq:zin_llc} is indeed independent of the choice of $\varphi_z$.

We will now argue that they are also independent of the choice of $z$-embedding. If $G \rw G_z$ and $G \rw G_y$ are two $z$-embeddings, we construct as in Fact \ref{fct:zin_filt} a common refinement $G_x$. Then $G \rw G_x$, $G_z \rw G_x$, and $G_y \rw G_x$, are pseudo-$z$-embeddings and the center of $G_x$ is connected. Choose $\varphi_x : L_F \rw {^LG_x}$ lifting $\varphi$ by Corollary \ref{cor:zin_liftpar} and let $\varphi_z : L_F \rw {^LG_z}$ and $\varphi_y : L_F \rw {^LG_y}$ be the corresponding parameters. We have the commutative diagrams of bijections
\[ \xymatrix@=1.5pc{
	&\pi_0(S_{\varphi_x}^+)\ar[dl]\ar[dr]&&&\Pi_{\varphi_x}(G_x')\ar[dl]\ar[dr]&\\
	\pi_0(S_{\varphi_y}^+)\ar[dr]&&\pi_0(S_{\varphi_z}^+\ar[dl])&\Pi_{\varphi_y}(G_y')\ar[dr]&&\Pi_{\varphi_z}(G_z')\ar[dl]\\
	&\pi_0(S_{\varphi}^+)&&&\Pi_{\varphi}(G')
}
\]
This together with the natural compatibility of $\tx{LLC}_\tx{rig}$ along the pseudo-$z$-embeddings $G_z \rw G_x$ and $G_y \rw G_x$ assumed above implies that the set $\Pi_\varphi(G')$ and the bijection \eqref{eq:zin_llc} provided by $G_z$ coincide with those provided by $G_y$.

The sets $\Pi_\varphi(G')$ for various $\varphi$ exhaust $\Pi_\tx{temp}(G')$. Indeed, for any $\pi \in \Pi_\tx{temp}(G')$ we can find an extension of its central character to a unitary character $\chi : Z(G_z')(F) \rw \C^\times$. Then $\pi_z = \pi \otimes \chi$ is an extension of $\pi$ to an element $\pi_z \in \Pi_\tx{temp}(G_z')$, which by $\tx{LLC}_\tx{rig}(\psi_z,x_\tx{rig})$ belongs to some packet $\Pi_{\varphi_z}(G_z')$. By construction $\pi$ then belongs to $\Pi_\varphi(G')$, where $\varphi$ is the composition of $\varphi_z$ with the projection $^LG_z \rw {^LG}$. The same argument also shows that the sets $\Pi_\varphi(G')$ for various parameters $\varphi$ are disjoint.

We will now argue that the character identity \eqref{eq:charidri} holds for the packet $\Pi_\varphi(G)$. It will be more convenient to consider the following equivalent formulation
\begin{equation} \label{eq:zin_charid}
\Theta^{\dot s}_{\varphi,x_\tx{rig}}(\delta') = \sum_{\gamma_1} \Delta'[\mf{w},\mf{\dot e},\mf{z},(\psi,x_\tx{rig})](\gamma_1,\delta')\Delta_{IV}(\gamma_1,\delta')^{-2}\Theta^1_{\xi_1\circ\varphi,1}(\gamma^1).
\end{equation}
Here $\delta' \in G'(F)$ is a strongly regular semi-simple element and $\gamma_1$ runs over the set of stable conjugacy classes of strongly regular semi-simple elements in $H_1(F)$.

According to the construction of $\Pi_\varphi(G)$, the virtual character $\Theta^{\dot s}_{\varphi,x_\tx{rig}}$ is the restriction to $G'(F)$ of the virtual character $\Theta^{\dot s}_{\varphi_z,x_\tx{rig}}$ of the group $G_z'(F)$ for any lift $\varphi_z$ of $\varphi$. In the same way, $\Theta^1_{\xi_1\circ\varphi,1}$ is the restriction to $H_1(F)$ of the virtual character $\Theta^1_{\xi_{z,1}\circ\varphi_z,1}$ of the group $H_{z,1}(F)$. Lemma \ref{lem:zin_tf} implies that the transfer factor $\Delta'[\mf{w},\mf{\dot e},\mf{z},(\psi,x_\tx{rig})]$ remains unchanged if we pass from $G'$ to $G'_z$. In the proof we mentioned the much simpler statement that the factor $\Delta_{IV}$ also doesn't change. Finally, the set of stable classes in $H_{z,1}(F)$ of the element $\gamma_1 \in H_1(F)$ is the same as the set of stable classes in $H_1(F)$. The identity \eqref{eq:zin_charid} thus follows from the corresponding identity for the parameter $\varphi_z$.

\section{Changing the rigidifying datum in $\tx{LLC}_\tx{rig}$} \label{sec:changrig}

In this section we will study the following question: Given an connected reductive group $G$ defined and quasi-split over $F$, an inner twist $\psi : G \rw G'$, and two elements $x_{1,\tx{rig}},x_{2,\tx{rig}} \in Z^1(u \rw W,Z(G) \rw G)$ lifting $\psi^{-1}\sigma(\psi) \in Z^1(\Gamma,G_\tx{ad})$, what is the relationship between the statements $\tx{LLC}_\tx{rig}(\psi,x_{1,\tx{rig}})$ and $\tx{LLC}_\tx{rig}(\psi,x_{2,\tx{rig}})$? The answer to this question will be given by an explicit relation between the two statements. This relation can be used either to switch from one normalization to another in applications, or to deduce the validity of one normalization from the validity of another as a step in the proof of $\tx{LLC}_\tx{rig}$. The latter situation will occur if one wants to deduce $\tx{LLC}_\tx{rig}$ from $\tx{LLC}_\tx{iso}$ using the results of the previous two sections, because not all element $x_\tx{rig}$ will come from $B(G)_\tx{bas}$ or $B(G_z)_\tx{bas}$. This situation would also occur if one wants to deduce $\tx{LLC}_\tx{rig}$ using the stabilized trace formula and the local-global passage established in \cite{KalGRI}.

\subsection{Description of $H^1(W,Z)$} \label{sub:tn+z}

In \cite{KalRI} we studied the cohomology set $H^1(u \rw W,Z \rw G)$, where $G$ is a connected reductive group, in particular a torus, and $Z$ is a finite central subgroup. In order to understand how $\tx{LLC}_\tx{rig}(\psi,x_\tx{rig})$ depends on the choice of $x_\tx{rig}$, we will also need to understand the cohomology group $H^1(u \rw W,Z \rw Z)$, where $Z$ is a finite (multiplicative) algebraic group defined over $F$. This cohomology group is the same as $H^1(W,Z)$ -- the group of continuous cohomology classes of the topological group $W$ with values in the finite group $Z(\ol{F})$.

Let $S$ be a torus over $F$, $Z \subset S$ a finite subgroup, and $\bar S=S/Z$. We write again $Y=X_*(S)$ and $\bar Y=X_*(\bar S)$. The following is part of \cite[Diagram (3.6)]{KalRI},
\begin{equation} \label{eq:bfd}
\xymatrix{
	1\ar[r]&H^1(\Gamma,Z)\ar[r]^-{\tx{Inf}}\ar[d]&H^1(u \rw W,Z \rw Z)\ar[r]^-{\tx{Res}}\ar[d]&\tx{Hom}(u,Z)^\Gamma\ar@{=}[d]\\
1\ar[r]&H^1(\Gamma,S)\ar[r]^-{\tx{Inf}}\ar@{=}[d]&H^1(u\rw W,Z\rw S)\ar[r]^-{\tx{Res}}\ar[d]&\tx{Hom}(u,Z)^\Gamma\ar[d]\\
&H^1(\Gamma,S)\ar[r]&H^1(\Gamma,\bar S)\ar[r]\ar[d]&H^2(\Gamma,Z)\ar[d]\\
&&1&1
}
\end{equation}
describing the relationship between $H^1(u \rw W,Z \rw S)$ and the usual cohomology groups $H^1(\Gamma,S)$ and $H^1(\Gamma,\bar S)$. We have so far the following diagram that is isomorphic to Diagram \eqref{eq:bfd}

\begin{equation} \label{eq:bfdtn}
\xymatrix{
	1\ar[r]&\hat H^{-2}(\Gamma,\bar Y/Y)\ar[r]\ar[d]&?\ar[r]\ar[d]&\hat Z^{-1}(\Gamma,\bar Y/Y)\ar@{=}[d]\\
	1\ar[r]&\hat H^{-1}(\Gamma,Y)\ar[r]\ar@{=}[d]&\frac{\hat Z^{-1}(\Gamma,\bar Y)}{\hat B^{-1}(\Gamma,Y)}\ar[r]\ar[d]&\hat Z^{-1}(\Gamma,\bar Y/Y)\ar[d]\\
	&\hat H^{-1}(\Gamma,Y)\ar[r]&\hat H^{-1}(\Gamma,\bar Y)\ar[r]&\hat H^{-1}(\Gamma,\bar Y/Y)\ar[d]\\
	&&&1
}
\end{equation}
Here we are using hats to denote Tate cohomology groups. Since $\Gamma$ is not a finite group, we must explain what we mean by that. We warn the reader that we do \emph{not} mean the Tate cohomology groups for profinite groups as defined for example in \cite[Ch 1. \S9]{NSW08}. Let $A$ be a discrete $\Gamma$-module that is finitely generated over $\Z$. Then it is inflated from $\Gamma_{E/F}$ for some finite Galois extension $E/F$. For any finite Galois extension $K/F$ containing $E$, the identity map $\hat C^{-1}(\Gamma_{E/F},A)=A=\hat C^{-1}(\Gamma_{K/F},A)$ respects the subgroups $\hat Z^{-1}$ and $\hat B^{-1}$ and hence produces a map $\hat H^{-1}(\Gamma_{E/F},A) \rw \hat H^{-1}(\Gamma_{K/F},A)$. We declare $\hat H^{-1}(\Gamma,A)$ to be the colimit of this system. It is easily seen that this colimit stabilizes. As for degree $(-2)$, assume further that $A$ is finite and define $\hat H^{-2}(\Gamma,A)$ to be equal to $H_1(\Gamma,A)$. This is the limit of the finite groups $\hat H^{-2}(\Gamma_{K/F},A)=H_1(\Gamma_{K/F},A)$ with respect to the coinflation map. It is argued in \cite[VI.1]{Lan83} that this limit stabilizes.

The isomorphism $\hat H^{-1}(\Gamma,Y) \rw H^1(\Gamma,S)$ and its analog for $\bar S$ are the usual Tate-Nakayama isomorphisms. The isomorphism $\frac{\hat Z^{-1}(\Gamma,\bar Y)}{\hat B^{-1}(\Gamma,Y)} \rw H^1(u \rw W,Z \rw S)$ was constructed in \cite[\S4]{KalRI}, where also the more elementary isomorphism $\hat Z^{-1}(\Gamma,\bar Y/Y) \rw \tx{Hom}(u,Z)^\Gamma$ is discussed. The isomorphisms $\hat H^{-2}(\Gamma,\bar Y/Y) \rw H^1(\Gamma,Z)$ and $\hat H^{-1}(\Gamma,\bar Y/Y) \rw H^2(\Gamma,Z)$ are variants of Poitou-Tate duality and are discussed in \cite[VI.1]{Lan83}, where it is also shown that the limit of $\hat H^{-2}$ stabilizes.

The purpose of this section is to demystify the question mark in Diagram \eqref{eq:bfdtn} and the arrows connecting with it. We claim that
\[ ?=\hat C^{-2}(\Gamma,\bar Y/Y)/\hat B^{-2}(\Gamma,\bar Y/Y) = \varprojlim\hat C^{-2}(\Gamma_{E/F},\bar Y/Y)/\hat B^{-2}(\Gamma_{E/F},\bar Y/Y), \]
where again the limit is taken over all finite Galois extensions $K/F$ through which the action of $\Gamma$ on $\bar Y/Y$ factors, and the transition maps are given by coinflation. In Diagram \eqref{eq:bfdtn}, the horizontal map going into this term is given by the natural inclusion of $\hat Z^{-2}$ into $\hat C^{-2}$, the horizontal map going out of this term is given by the differential, and the vertical map going out of this term is given by first lifting an element of $?$ to an element of $C^{-2}(\Gamma_{E/F},\bar Y)$ and then taking the differential. A simple computation shows that all these maps respect the relevant transition maps in the direct and inverse systems involved and that the diagram is commutative. It is also clear that the outer rim of that diagram, i.e. the sequence
\begin{equation} \label{eq:zz}
1\rw\hat H^{-2}(\Gamma,\bar Y/Y)\rw\frac{\hat C^{-2}(\Gamma,\bar Y/Y)}{\hat B^{-2}(\Gamma,\bar Y/Y)}
\rw\hat Z^{-1}(\Gamma,\bar Y/Y)
\rw\hat H^{-1}(\Gamma,\bar Y/Y)\rw1,
\end{equation}
is exact. The latter corresponds to the exactness of the part of Diagram \eqref{eq:bfd} corresponding to the inf-res sequence \cite[(3.5)]{KalRI} for $G=Z$.

Recall $\hat C^{-2}(\Gamma_{E/F},\bar Y/Y)=\tx{Maps}(\Gamma_{E/F},\bar Y/Y)$ and that given such a $(-2)$\-cochain $f$, its differential is $df = \sum_{\sigma \in \Gamma_{E/F}} \sigma^{-1}f(\sigma)-f(\sigma) \in \hat C^{-1}(\Gamma_{E/F},\bar Y/Y)=\bar Y/Y$. Recall further that the coinflation map sends $f' \in \tx{Maps}(\Gamma_{K/F},\bar Y/Y)$ to $f \in \tx{Maps}(\Gamma_{E/F},\bar Y/Y)$ given by $f(\sigma)=\sum_{\sigma' \mapsto \sigma} f'(\sigma')$. From this formula it is obvious that coinflation is surjective. Moreover, since the first, third, and fourth (co)limits of the above four-term exact sequence all stabilize, so must also the second term.

Next we define an isomorphism
\begin{equation} \label{eq:tn+z}
	\hat C^{-2}(\Gamma,\bar Y/Y)/\hat B^{-2}(\Gamma,\bar Y/Y) \rw H^1(W,Z).
\end{equation}
We choose $S$ so that $H^1(\Gamma,\bar S)=1$ and the map $H^1(\Gamma,Z) \rw H^1(\Gamma,S)$ is bijective. This is possible according to Proposition \ref{pro:zembex}. We claim that then the map $H^1(W,Z) \rw H^1(u \rw W,Z \rw S)$ is also bijective. Indeed, its surjectivity is immediate from $H^1(\Gamma,\bar S)=1$. Its kernel is equal to the image of $\bar S(F)=\bar S(\ol{F})^W$ in $H^1(W,Z)$ under the connecting homomorphism. This is the same as the inflation of the kernel of $H^1(\Gamma,Z) \rw H^1(\Gamma,S)$, which is trivial. A similar argument shows that the map $\hat C^{-2}(\Gamma,\bar Y/Y)/\hat B^{-2}(\Gamma,\bar Y/Y) \rw \hat Z^{-1}(\Gamma,\bar Y)/\hat B^{-1}(\Gamma,Y)$ is bijective. We now define \eqref{eq:tn+z} as the composition of the three bijections
\begin{equation} \label{eq:tn+z-def} \frac{\hat C^{-2}(\Gamma,\bar Y/Y)}{\hat B^{-2}(\Gamma,\bar Y/Y)} \rw \frac{\hat Z^{-1}(\Gamma,\bar Y)}{\hat B^{-1}(\Gamma,Y)} \rw H^1(u \rw W,Z \rw S) \rw H^1(W,Z). \end{equation}
We must now argue that this composition is independent of the choice of $S$ and is functorial in $Z$. For independence of $S$, let $Z \rw S_1$ and $Z \rw S_2$ be two choices of $S$. Let $S_3$ be the push-out $S_1 \times_Z S_2$. Then $Z \rw S_3$ given by $z \mapsto (z,1)=(1,z)$ is a third embedding with the same properties. Moreover, we have the embeddings $S_1 \rw S_3$ and $S_2 \rw S_3$ given by $s_1 \mapsto (s_1,1)$ and $s_2 \mapsto (1,s_2)$. The first one leads to the exact sequence
\[ 1 \rw S_1 \rw S_3 \rw \bar S_2 \rw 1\]
from which, by taking $W$-cohomology, we obtain the exact sequence
\[ S_3(F) \rw \bar S_2(F) \rw H^1(u \rw W,Z \rw S_1) \rw H^1(u \rw W,Z \rw S_2 ) \rw H^1(\Gamma,\bar S_2). \]
Now $S_3(F)$ contains $S_2(F)$ which surjects onto $\bar S_2(F)$, while $H^1(\Gamma,\bar S_2)=1$, and we conclude that $H^1(u \rw W,Z \rw S_1) \rw H^1(u \rw W,Z \rw S_3)$ is bijective. In the same way we conclude that $H^1(u \rw W,Z \rw S_2) \rw H^1(u \rw W,Z \rw S_3)$ is bijective. This, together with the fact that all maps in \eqref{eq:tn+z-def} are functorial in $S$ proves that $S_1$ and $S_2$ lead to the same isomorphism \eqref{eq:tn+z}.

The proof of functoriality of \eqref{eq:tn+z} in $Z$ is similar. Given $Z_1 \rw Z_2$, choose embeddings $Z_1 \rw S_1$ and $Z_2 \rw S_2$ and form $S_3=S_1 \times_{Z_1} S_2$. Then we obtain the exact sequence
\[ 1 \rw Z_2 \rw S_3 \rw \bar S_1 \times \bar S_2 \rw 1. \]
The map $S_3(F) \rw \bar S_2(F) \times \bar S_1(F)$ is surjective, because its composition with the obvious map $S_1(F) \times S_2(F) \rw S_3(F)$ gives the surjective map $S_1(F) \times S_2(F) \rw \bar S_1(F) \times \bar S_2(F)$. Moreover, $H^1(\Gamma,\bar S_1 \times \bar S_2)=1$. Thus we may construct the isomorphism \eqref{eq:tn+z} for $Z_2$ by using \eqref{eq:tn+z-def} with the embedding $Z_2 \rw S_3$. But we now have the morphism $[Z_1 \rw S_1] \rw [Z_2 \rw S_3]$ of embeddings and the functoriality of \eqref{eq:tn+z} follows from the functoriality of the three arrows in \eqref{eq:tn+z-def}.

Although we will not need this, we remark that there is an explicit formula for the isomorphism \eqref{eq:tn+z} that does not involve a choice of $S$. In order to give it, we will use the notation established in \cite[\S4.4,\S4.5,\S4.6]{KalRI}. In particular we have the exhaustive tower of finite Galois extensions $E_k/F$, a co-final sequence $n_k \in \N$ (which we may specify to be $n_k=[E_k:F]$), a 1-cocycle $c_k \in Z^2(\Gamma_{E_k/F},E_k^\times)$ representing the fundamental class, and an $n_k$-th root map $l_k : \ol{F}^\times \rw \ol{F}^\times$. This data leads to explicit elements $\xi_k \in Z^2(\Gamma,u_{E_k/F,n_k})$ given by \cite[(4.7)]{KalRI} and thus to explicit extensions $W_k = u_{E_k/F,n_k} \boxtimes_{\xi_k} \Gamma$. There are also surjective transition maps $W_{k+1} \rw W_k$ and the limit of this system is $W$.

In order to give the intrinsic formula of \eqref{eq:tn+z}, we first replace the finite $\Gamma$-module $\bar Y/Y$, which a-priori depends on the choice of $S$, by the isomorphic module $A^\vee = \tx{Hom}(X^*(Z),\Q/\Z)$. Let $[\bar\lambda] \in \hat C^{-2}(\Gamma,A^\vee)$. Choose $k$ large enough so that $\exp(Z)|n_k$ and let  $[\bar\lambda_k] \in \hat C^{-2}(\Gamma_{E_k/F},A^\vee)$ be the image of $[\bar\lambda]$. Then the map
\[ z_{[\bar\lambda_k]} : W_k \rw Z(\ol{F}),\quad x \boxtimes \sigma \mapsto \phi_{d[\bar\lambda_k],k}(x) \cdot (-dl_kc_k \sqcup_{E_k/F}[\bar\lambda_k])(\sigma) \]
is an element of $Z^1(W_k,Z)$. In the first factor on the right we are using the isomorphism $[A^\vee]^{N_{E_k/F}} \rw \tx{Hom}(u_{E_k/F,n_k},Z)^\Gamma$ discussed in the beginning of \cite[\S4.6]{KalRI} to obtain $\phi_{d[\bar\lambda_k],k}$. In the second factor on the right $\sqcup_{E_k/F}$ is the unbalanced cup-product of \cite[\S4.3]{KalRI} and we are using the isomorphism $A^\vee=\tx{Hom}(\mu_{n_k},Z)$.  One can check that the inflation of the class $[z_{[\bar\lambda_k]}]$ to an element $[z_{[\bar\lambda]}] \in H^1(W,Z)$ is independent of the choice of $k$ and that $[\bar\lambda] \mapsto [z_{[\bar\lambda]}]$ is an explicit realization of \eqref{eq:tn+z}. As we will not need this statement, we will not give a justification.

\subsection{From isomorphism to duality} \label{sub:tnd}
In this section we are going to explicitly describe the Pontryagin dual of the commutative diagram \eqref{eq:bfdtn}.

Let again $S$ be a torus over $F$, $Z \subset S$ a finite subgroup, and $\bar S=S/Z$. Let $Y=X_*(S)$ and $\bar Y=X_*(\bar S)$. We consider the dual tori $\hat S=\tx{Hom}(Y,\C^\times)$ and $\hat{\bar S}=\tx{Hom}(\bar Y,\C^\times)$. They form the isogeny
\[ 1 \rw \hat Z \rw \hat{\bar S} \rw \hat S \rw 1 \]
with $\hat Z=\tx{Hom}(\bar Y/Y,\C^\times)$ being the Pontryagin dual of $Z$. We claim that the Pontryagin dual of Diagram \eqref{eq:bfdtn} is the following diagram.

\begin{equation} \label{eq:bfdpd}
\xymatrix{
	1&\ar[l]H^1(\Gamma,\hat Z)&\ar[l]Z^1(\Gamma,\hat Z)\ar[l]&\ar[l]_-{-d}\frac{\hat C^0(\Gamma,\hat Z)}{\hat B^0(\Gamma,\hat Z)}\\
	1&\ar[l]\hat H^0(\Gamma, \hat S)\ar[u]^{-\delta}&\ar[l]\pi_0([\hat{\bar S}]^+)\ar[u]^{-d}&\ar[l]\frac{\hat C^0(\Gamma,\hat Z)}{\hat B^0(\Gamma,\hat Z)}\ar@{=}[u]\\
	&\hat H^0(\Gamma,\hat S)\ar@{=}[u]&\ar[l]\hat H^0(\Gamma,\hat{\bar S})\ar[u]&\ar[l]\hat H^0(\Gamma,\hat Z)\ar[u]
}
\end{equation}
Here we have defined $\hat H^0(\Gamma,-)$ in the same way as we defined $\hat H^{-1}(\Gamma,-)$ in the previous section -- as the colimit with respect to the transition maps induced by the identity $\hat C^0(\Gamma_{E/F},A)=A=\hat C^0(\Gamma_{K/F},A)$ for any tower of finite Galois extensions $K/E/F$ with $A$ being inflated from $\Gamma_{E/F}$.

In the middle term of the diagram $[\hat{\bar S}]^+$ is again the preimage in $\hat{\bar S}$ of $\hat S^\Gamma$. We have written $d$ for the obvious differentials, and $\delta$ for the connecting homomorphism. We are forced to place minus signs in order to obtain the correct duality, as we shall now see.

To describe how each term of this diagram is the Pontryagin dual of the corresponding term in Diagram \eqref{eq:bfdtn}, we begin with the term involving $\hat S$ and $\hat{\bar S}$. We have the natural pairing $Y \otimes \hat S \rw \C^\times$. If we write $N$ for the norm map of the action of $\Gamma_{E/F}$, where $E/F$ is any finite Galois extension splitting $S$, then the kernel of $N$ in $Y$ is the exact annihilator of the image of $N$ in $\hat S$, the latter happening to be $\hat S^{\Gamma,\circ}$.
At the same time, $I_{E/F}Y$ is the exact annihilator of $\hat S^\Gamma$. This explains why the bottom left square in Diagram \eqref{eq:bfdpd} is dual to the bottom left square in Diagram \eqref{eq:bfdtn}.
To describe the terms involving $\hat Z$, we use the following.

\begin{lem}
Let $\Delta$ be a finite group and let $A$ and $B$ be finite $\Delta$-modules in duality. For any $i \geq 0$, cup product induces perfect duality of finite groups
\[ \hat C^{-i-1}(\Delta,A) \otimes \hat C^i(\Delta,B) \rw \hat C^{-1}(\Delta,\Q/\Z)=\Q/\Z \]
under which $\hat Z^{-i-1}(\Delta,A)^\perp=\hat B^i(\Delta,B)$ and $\hat B^{-i-1}(\Delta,A)^\perp=\hat Z^i(\Delta,B)$. If $\Delta' \rw \Delta$ is a surjection of finite groups, then
\[ a \cup \tx{inf}(b) = \tx{inf}(\tx{coinf}(a) \cup b),\qquad a \in \hat C^{-i-1}(\Delta',A), b \in \hat C^i(\Delta,B). \]
\end{lem}
\begin{proof}

	The perfect duality and the compatibility with inflation and coinflation can be seen by a direct computation using the formula for the cup product. The statement about annihilators comes from the formula $da\cup b + (-1)^{i+1}a \cup db = d(a \cup b)$ and the fact that $\hat B^{-1}(\Delta,\Q/\Z)=0$.
\end{proof}

This lemma shows that each term in Diagram \eqref{eq:bfdpd} involving $\hat Z$ is dual to the corresponding term in Diagram \eqref{eq:bfdtn}. To show commutativity, we reinterpret the natural pairing $Y \otimes \hat S \rw \C^\times$ as the pairing $\hat C^{-1}(\Gamma_{E/F},Y) \otimes \hat C^0(\Gamma_{E/F},\hat S) \rw \hat C^{-1}(\Gamma_{E/F},\C^\times)$ given by the cup-product, where $E/F$ is any finite Galois extension splitting $S$. The commutativity of the diagram now follows from the formula $da\cup b + (-1)^{i+1}a \cup db = d(a \cup b)$ and the fact that $\hat B^{-1}(\Delta,\Q/\Z)=0$.

\subsection{Switching between normalizations}
We will now discuss the effect of changing the rigidifying element $x_\tx{rig}$ of the rigid inner twist $(\psi,x_\tx{rig})$ on the statement $\tx{LLC}_\tx{rig}(\psi,x_\tx{rig})$. Let $\psi : G \rw G'$ be an inner twist and $x_{1,\tx{rig}},x_{2,\tx{rig}} \in Z^1(u \rw W,Z(G) \rw G)$ be two elements lifting $\psi^{-1}\sigma(\psi) \in Z^1(\Gamma,G_\tx{ad})$. Given a tempered Langlands parameter $\varphi : L_F \rw {^LG}$ we have the statements $\tx{LLC}_\tx{rig}(\psi,x_{1,\tx{rig}})$ and $\tx{LLC}_\tx{rig}(\psi,x_{2,\tx{rig}})$, each of which leads to one of the two bijections
\[ \tx{Irr}(\pi_0(S_\varphi^+),[x_{1,\tx{rig}}]) \rw \Pi_\varphi(G') \lw \tx{Irr}(\pi_0(S_\varphi^+),[x_{2,\tx{rig}}]). \]
We will now describe an explicit bijection
\begin{equation} \label{eq:llcrigy} \tx{Irr}(\pi_0(S_\varphi^+),[x_{1,\tx{rig}}]) \rw  \tx{Irr}(\pi_0(S_\varphi^+),[x_{2,\tx{rig}}]) \end{equation}
and then argue that this bijection is compatible with the above two bijections in the obvious way.

For this let $n$ be large enough so that $x_{1,\tx{rig}},x_{2,\tx{rig}} \in Z^1(u \rw W,Z_n \rw G)$. Then there exists $y \in Z^1(W,Z_n)$ with $x_{2,\tx{rig}}=y\cdot x_{1,\tx{rig}}$. We have the exact sequence
\[ 1 \rw \hat Z_n \rw \hat G_n \rw \hat G \rw 1, \]
where we have defined $\hat Z_n$ to be the kernel of the projection $\hat G_n \rw \hat G$, which is at the same time the Pontryagin dual of $Z_n$. On this sequence we have an action of $L_F$ via $\tx{Ad}\circ\varphi$. Since each element of $\tx{Irr}(\pi_0(S_\varphi^+),[x_{i,\tx{rig}}])$ kills the kernel of $\hat{\bar G} \rw \hat G_n$, we may replace $S_\varphi^+$ with its image in $\hat G_n$, which we do without change in notation. The differential $d : C^0(L_F,\hat G_n) \rw C^1(L_F,\hat G_n)$, when restricted to the subgroup $S_\varphi^+$, factors through $\pi_0(S_\varphi^+)$ and takes image in $Z^1(L_F,\hat Z_n)$. The action of $L_F$ on $\hat Z$ by $\tx{Ad}\circ\varphi$ is inflated from $W_F$ and is the same as the action of $W_F$ on $\hat Z_n$ coming from the $\Gamma$-structure on $\hat{\bar G}$. Moreover, since $\hat Z_n$ is finite, we have $Z^1(L_F,\hat Z_n)=Z^1(W_F,\hat Z_n)=Z^1(\Gamma,\hat Z_n)$. The element $[y] \in H^1(W,Z_n)$ provides a character on $Z^1(\Gamma,\hat Z_n)$ as discussed in Subsection \ref{sub:tnd}. Via the negative differential $-d$, we pull this character to a linear character $\pi_0(S_\varphi^+) \rw \C^\times$. The bijection \eqref{eq:llcrigy} is given by tensor product with this linear character.

\begin{lem} \label{lem:lpackca} Assume the validity of $\tx{LLC}_\tx{rig}(\psi,x_{2,\tx{rig}})$. Then the composition of the bijection $\tx{Irr}(\pi_0(S_\varphi^+),[x_{2,\tx{rig}}]) \rw \Pi_\varphi(G')$ with the bijection \eqref{eq:llcrigy} is the unique bijection $\tx{Irr}(\pi_0(S_\varphi^+),[x_{1,\tx{rig}}]) \rw \Pi_\varphi(G')$ that makes $\tx{LLC}_\tx{rig}(\psi,x_{1,\tx{rig}})$ true.
\end{lem}
\begin{proof}
Consider the left hand side of \eqref{eq:charidri} for the two rigid inner twists $(\psi,x_{1,\tx{rig}})$ and $(\psi,x_{2,\tx{rig}})$. Let us denote the two functions occurring there by $f^{\mf{\dot e},1}$ and $f^{\mf{\dot e},2}$. The condition of matching orbital integrals together with Lemma \ref{lem:tfca} imply $f^{\mf{\dot e},2}=\<[y],(-d)\dot s\>f^{\mf{\dot e},1}$. Looking at the right hand side of \eqref{eq:charidri} and its definition \eqref{eq:scharri} we conclude that $\tx{LLC}_\tx{rig}(\psi,x_{1,\tx{rig}})$ is equivalent to the equation
\[ \sum_{\pi \in \Pi_\varphi(G')} \<\dot\pi,\dot s\>_2\Theta_{\dot\pi} = \<[y],(-d)\dot s\> \sum_{\pi \in \Pi_\varphi(G')} \<\dot\pi,\dot s\>_1 \Theta_{\dot\pi}, \]
where we have inserted the subscripts 1 and 2 to distinguish between the two pairings coming from the two statements $\tx{LLC}_\tx{rig}(\psi,x_{1,\tx{rig}})$ and $\tx{LLC}_\tx{rig}(\psi,x_{2,\tx{rig}})$. The linear independence of the characters of tempered representations of $G'(F)$ imply $\<\dot\pi,\dot s\>_2=\<[y],(-d)\dot s\>\<\pi,\dot s\>_1$. Since this is true for all $\dot s \in \pi_0(S_\varphi^+)$ we are done.
\end{proof}

In order to complete the proof of Lemma \ref{lem:lpackca} we must state and prove Lemma \ref{lem:tfca}, which tells us how the transfer factor \eqref{eq:ritf} changes when we switch from $x_{1,\tx{rig}}$ to $x_{2,\tx{rig}}$. For this we take a second look at the complex number $\<[y],(-d)\dot s\>$. It can be reinterpreted as follows. Let $\mf{\dot e}=(H,\mc{H},\dot s,\xi)$ be the refined endoscopic datum associated to $(\varphi,\dot s)$, as explained in Subsection \ref{sub:reviewllc}. We map $Z_n$ under $Z(G) \rw Z(H)$ and form $\bar H=H/Z_n$. If we restrict the differential $d : C^0(\Gamma,Z(\hat{\bar H})) \rw C^1(\Gamma,Z(\hat{\bar H}))$ to the subgroup $Z(\hat{\bar H})^+ \subset C^0(\Gamma,Z(\hat{\bar H}))$, then it kills the connected component $Z(\hat{\bar H})^{+,\circ}$ (because it is a subgroup of $Z(\hat{\bar H})^\Gamma$) and its image belongs to $Z^1(\Gamma,\hat Z_n)$. We can thus map the element $\dot s \in \pi_0(Z(\hat{\bar H})^+)$ of the refined endoscopic datum $\mf{\dot e}$ under \emph{the negative} of this differential and obtain an element $(-d)\dot s \in Z^1(\Gamma,\hat Z_n)$. We can then pair this element with the class of $y$ in $H^1(W,Z_n)$ using the duality discussed in Section \ref{sub:tnd} and obtain the complex number $\<[y],(-d)\dot s\>$. Of course, this coincides with the previous definition of $\<[y],(-d)\dot s\>$, but this interpretation makes the following lemma independent of the previous discussion.

\begin{lem} \label{lem:tfca} We have
\[ \Delta'[\mf{w},\mf{\dot e},\mf{z},(\psi,x_{2,\tx{rig}})]  = \<[y],(-d)\dot s\>\Delta'[\mf{w},\mf{\dot e},\mf{z},(\psi,x_{1,\tx{rig}})].\]
\end{lem}
\begin{proof}
Let $\gamma_1 \in H_{1,G\tx{-sr}}(F)$ and $\delta' \in G_\tx{sr}'(F)$ be a pair of related elements and fix $\delta \in G_\tx{sr}(F)$ stably conjugate to $\delta'$. Then according to \eqref{eq:ritf} we have
\[
\frac{\Delta'[\mf{w},\mf{\dot e},\mf{z},(\psi,x_{2,\tx{rig}})](\gamma_1,\delta')}{\Delta'[\mf{w},\mf{\dot e},\mf{z},(\psi,x_{1,\tx{rig}})](\gamma_1,\delta')} = \frac{\<\tx{inv}[x_{2,\tx{rig}}](\delta,\delta'),\dot s_{\gamma,\delta}\>}{\<\tx{inv}[x_{1,\tx{rig}}](\delta,\delta'),\dot s_{\gamma,\delta}\>}.
\]
Both $\tx{inv}[x_{2,\tx{rig}}](\delta,\delta')$ and $\tx{inv}[x_{1,\tx{rig}}](\delta,\delta')$ are elements of $H^1(u \rw W,Z_n \rw S)$, where $S \subset G$ is the centralizer of $\delta$. The difference $\tx{inv}[x_{2,\tx{rig}}](\delta,\delta') - \tx{inv}[x_{1,\tx{rig}}](\delta,\delta')$ is equal to the image of $[y] \in H^1(W,Z_n)$. It follows that the right hand side above is equal to $\<[y],\dot s_{\gamma,\delta}\>$. According to Diagram \eqref{eq:bfdpd}, we would get the same result if we mapped $\dot s_{\gamma,\delta} \in \pi_0([\hat{\bar S}]^+)$ to the group $Z^1(\Gamma,\hat Z_n)$ via $-d$ and then paired the result with $[y] \in H^1(W,Z_n)$. The image of $\dot s_{\gamma,\delta}$ in $Z^1(\Gamma,\hat Z_n)$ is the same as the image of $\dot s$ under the differential $-d : \pi_0(Z(\hat{\bar H})^+) \rw Z^1(\Gamma,\hat Z_n)$ and we see that the right hand side is equal to $\<[y],(-d)\dot s\>$, as claimed.
\end{proof}

\bibliographystyle{amsalpha}
\bibliography{/Users/tashokaletha/Dropbox/Work/TexMain/bibliography.bib}

\end{document}